\documentclass[11pt]{article}

\usepackage{amsmath,amsthm,amsfonts,amssymb,latexsym,amscd,enumerate,xypic}
\usepackage{extarrows, comment}

\theoremstyle{plain}
    \newtheorem{theorem}{Theorem}[section]
    \newtheorem{lemma}[theorem]{Lemma}
    \newtheorem{corollary}[theorem]{Corollary}
    \newtheorem{proposition}[theorem]{Proposition}

 \theoremstyle{definition}
    \newtheorem{definition}[theorem]{Definition}
    
    \newtheorem{example}[theorem]{Example}

    \newtheorem{remark}[theorem]{Remark}

\theoremstyle{remark}

\numberwithin{equation}{section}

\DeclareMathOperator{\Ad}{Ad}
\DeclareMathOperator{\ad}{ad}
\DeclareMathOperator{\indx}{index}

\DeclareMathOperator{\End}{End}
\DeclareMathOperator{\Hom}{Hom}
\DeclareMathOperator{\sgn}{sgn}

\DeclareMathOperator{\Spin}{Spin}

\DeclareMathOperator{\SO}{SO}
\DeclareMathOperator{\SU}{SU}
\DeclareMathOperator{\SL}{SL}

\DeclareMathOperator{\U}{U}

\DeclareMathOperator{\temp}{temp}
 \DeclareMathOperator{\Ind}{Ind}

\begin{document}

\newcommand{\Spinc}{\Spin^c}

    \newcommand{\R}{\mathbb{R}}
    \newcommand{\C}{\mathbb{C}} 
    \newcommand{\N}{\mathbb{N}}
    \newcommand{\Z}{\mathbb{Z}} 
    \newcommand{\Q}{\mathbb{Q}}
    \newcommand{\bK}{\mathbb{K}}

\newcommand{\g}{\mathfrak{g}}
\newcommand{\ka}{\mathfrak{a}}
\newcommand{\km}{\mathfrak{m}}
\newcommand{\kn}{\mathfrak{n}}
\newcommand{\kg}{\mathfrak{g}} 
\newcommand{\kt}{\mathfrak{t}}
\newcommand{\kA}{\mathfrak{A}}
\newcommand{\XX}{\mathfrak{X}}
\newcommand{\kh}{\mathfrak{h}} 
\newcommand{\kp}{\mathfrak{p}}
\newcommand{\p}{\mathfrak{p}}
\newcommand{\kk}{\mathfrak{k}}
\newcommand{\ks}{\mathfrak{s}}

\newcommand{\cE}{\mathcal{E}}
\newcommand{\cA}{\mathcal{A}}
\newcommand{\calL}{\mathcal{L}}
\newcommand{\calH}{\mathcal{H}}
\newcommand{\cO}{\mathcal{O}}
\newcommand{\cB}{\mathcal{B}}
\newcommand{\cK}{\mathcal{K}}
\newcommand{\cP}{\mathcal{P}}
\newcommand{\calD}{\mathcal{D}}
\newcommand{\cF}{\mathcal{F}}
\newcommand{\calR}{\mathcal{R}}
\newcommand{\cX}{\mathcal{X}}
\newcommand{\calM}{\mathcal{M}}
\newcommand{\cS}{\mathcal{S}}
\newcommand{\cU}{\mathcal{U}}

\newcommand{\Sj}{ \sum_{j = 1}^{\dim G}}
\newcommand{\Sk}{ \sum_{k = 1}^{\dim M}}
\newcommand{\ii}{\sqrt{-1}}

\newcommand{\Bigwedge}{\textstyle{\bigwedge}}

\newcommand{\ddt}{\left. \frac{d}{dt}\right|_{t=0}}

\newcommand{\mattwo}[4]{
\left( \begin{array}{cc}
#1 & #2 \\ #3 & #4
\end{array}
\right)
}

\newcommand{\PM}{P}
\newcommand{\DM}{D}
\newcommand{\LM}{L}
\newcommand{\vM}{v}

\newcommand{\Wedge}{\lambda}

\newcommand{\specialin}{\hspace{-1mm} \in \hspace{1mm} }

\newcommand{\bspl}{\[ \begin{split}}
\newcommand{\espl}{\end{split} \]}

\newcommand{\Utilde}{\widetilde{U}}
\newcommand{\Xtilde}{\widetilde{X}}
\newcommand{\Dtilde}{\widetilde{D}}
\newcommand{\Etilde}{\widetilde{S}}
\newcommand{\wt}{\widetilde}
\newcommand{\pd}{\overline{\partial}}

\newcommand{\Rhat}{\widehat{R}}

\newcommand{\beq}[1]{\begin{equation} \label{#1}}
\newcommand{\eeq}{\end{equation}}

\newenvironment{proofof}[1]
{\noindent \emph{Proof of #1.}}{\hfill $\square$}

\title{A geometric realisation of tempered representations restricted to maximal compact subgroups}

\author{Peter Hochs, Yanli Song and Shilin Yu}
\date{\today}

\maketitle

\begin{abstract}
Let $G$ be a connected, linear, real reductive Lie group with compact centre. Let $K<G$ be maximal compact. For a tempered representation $\pi$ of $G$, we realise the restriction $\pi|_K$ as the $K$-equivariant index of a Dirac operator on a homogeneous space of the form $G/H$, for a Cartan subgroup $H<G$. (The result in fact applies to every standard representation.) Such a space can be identified with a coadjoint orbit of $G$, so that we obtain an explicit version of Kirillov's orbit method for $\pi|_K$. In a companion paper, we use this realisation of $\pi|_K$ to give a geometric expression for the multiplicities of the $K$-types of $\pi$, in the spirit of the quantisation commutes with reduction principle. This generalises work by Paradan for the discrete series to arbitrary tempered representations.
\end{abstract}

\tableofcontents


\section{Introduction}

Let $G$ be a connected, linear, real reductive Lie group with compact centre, and let $\kg$ be its Lie algebra. Let $K<G$ be a maximal compact subgroup. Harish--Chandra showed that a unitary irreducible representation $\pi$ of $G$ is determined by the corresponding actions by $K$ and $\kg$ on the $K$-finite vectors in the representation space of $\pi$. This means that, in a sense, half the information about $\pi$ is contained in the restriction $\pi|_K$. The explicit form of this information consists of the multiplicities of the irreducible representations of $K$ in $\pi|_K$; i.e.\ the multiplicities of the $K$-types of $\pi$.

In this paper, we consider tempered representations $\pi$, and realise $\pi|_K$ as the equivariant index of a Dirac operator on a homogeneous space of $G$. In \cite{PHYSSY2}, we use this realisation to obtain a geometric expression for the multiplicities of the $K$-types of $\pi$. This expression is the index of a Dirac operator on a compact orbifold. These orbifolds are \emph{reduced spaces} in the $\Spinc$-sense,  analogous to those in symplectic geometry. The expression for multiplicities of $K$-types is an instance of the \emph{quantisation commutes with reduction} principle. This has direct consequences to the behaviour of multiplicities of $K$-types, such as criteria for them to equal $0$ or $1$.

Paradan \cite{Paradan03} did all of this for the discrete series. This paper is inspired by his work, and extends it to general tempered representations.

\subsection{Background and motivation}

Kirillov's orbit method is the idea that there should be a correspondence between (some) unitary irreducible representations $\pi$ of a Lie group $G$, and (some) orbits $\cO$ of the coadjoint action by $G$ on the dual $\kg^*$ of its Lie algebra. See \cite{Vogan00} for an account of the orbit method for reductive groups. On an intuitive level, the correspondence between orbits $\cO$ and representations $\pi$ is that $\pi = Q(\cO)$, the \emph{geometric quantisation} of $\cO$. In the modern mathematical approach, $Q(\cO)$ should be the equivariant index of a Dirac operator on $\cO$. But defining this rigorously is a challenge if $G$ and $\cO$ are noncompact.

A rigorous construction is essential for many applications, however. While the orbit method is a powerful guiding principle, it needs to be made precise for specific classes of groups and representations to have explicit consequences. The scope for applications depends on the properties of the construction $\pi = Q(\cO)$. It is very useful if $Q(\cO)$ is the index of an elliptic differential operator, because this opens up the possibility to apply powerful techniques from index theory to study $\pi$. For example, if the index used satisfies a fixed point formula, then this can be used to compute the global character of $\pi$. See \cite{HochsWang17} for an application of this to discrete series representations. If the index satisfies the quantisation commutes with reduction principle, then this can be used to express the decomposition into irreducibles of the restriction of $\pi$ to closed subgroups of $G$ in terms of the geometry of $\cO$.

Atiyah--Schmid \cite{Atiyah77}, Parthasarathy \cite{Parthasarathy72} and Schmid \cite{Schmid76} realised discrete series representations as $L^2$-kernels of Dirac operators on homogeneous spaces of $G$. Schmid's result fits directly into the orbit method framework. However, since kernels of Dirac operators are used, rather than indices, one cannot apply index theory  results, such as fixed point formulas and the quantisation commutes with reduction principle, to these realisations to obtain information about the discrete series. Paradan \cite{Paradan03} realised restrictions of discrete series representations $\pi$ to $K$ as $K$-equivariant indices of Dirac operators on coadjoint orbits $\cO$. While his realisation only applies to the restriction of  $\pi$ to $K$, his approach has the important advantage that the index he used satisfies the quantisation commutes with reduction principle. This principle, proved by Paradan in this setting, implies a geometric expression for the multiplicities of the $K$-types of $\pi$, as indices of Dirac operators on compact orbifolds. These orbifolds are \emph{reduced spaces} for the action by $K$ on $\cO$. If $p\colon \kg^* \to \kk^*$ is the restriction map, then the reduced space at $\xi \in \kk^*$ is
\[
\cO_{\xi} = (p^{-1}(\Ad^*(K)\xi) \cap \cO)/K.
\]

For discrete series representations $\pi$, Blattner's formula, proved by Hecht--Schmid \cite{HechtSchmid75}, is an explicit combinatorial expression for multiplicities of $K$-types. This was in fact used by Paradan to obtain his realisation of $\pi|_K$, and the resulting geometric multiplicity formula. But this geometric multiplicity formula has the advantage that it allows one to draw conclusions about the $K$-types of $\pi$ from the geometry of the corresponding coadjoint orbit. For example, about the question when their multiplicities equal $0$ or $1$.

In this paper, we generalise Paradan's construction to arbitrary tempered representations. Tempered representations are those unitary irreducible representations whose $K$-finite matrix coefficients are in $L^{2+\varepsilon}(G)$, for all $\varepsilon > 0$. 
The set $\hat G_{\temp}$ of these representations occurs in the Plancherel decomposition
\[
L^2(G) = \int_{\hat G_{\temp}}^{\oplus} \pi \otimes \pi^*\, d\mu(\pi)
\]
of $L^2(G)$ as a representation of $G\times G$. Here $\mu$ is the Plancherel measure. For this reason, tempered representations are central to harmonic analysis. Furthermore, they feature in the Langlands classification of admissible irreducible representations.

For general tempered representations, a multiplicity formula for $K$-types is especially valuable, because no explicit formula multiplicity formula exists yet in this generality. There are algorithms to compute these multiplicities, see the ATLAS software package\footnote{See \texttt{http://www.liegroups.org/software/}.}  developed by du Cloux, van Leeuwen, Vogan and many others; see also \cite{ATLASdoc}. But it is a challenge to draw conclusions about the general behaviour of multiplicities of $K$-types from these algorithms. The main difficulty is that they involve representations of disconnected subgroups, which cannot be classified via Lie algebra techniques. Also, already for the discrete series,  cancellation of terms in Blattner's formula can make it nontrivial to evaluate it, and for example to see when multiplicities equal zero.
The geometric multiplicity formula we deduce from the main result of this paper in \cite{PHYSSY2} allows us to use the geometry of coadjoint orbits to study multiplicities of $K$-types. In that paper, we obtain applications to multiplicity-free restrictions, and for example show that tempered representations of $\SU(p,1)$, $\SO_0(p,1)$ and $\SO_0(2,2)$ with regular infinitesimal characters have multiplicity-free restrictions to maximal compact subgroups.

\subsection{The main result}

Let $\pi$ be a tempered representation of $G$. The main result in this paper is Theorem \ref{main theorem}, a realisation of $\pi|_K$ as the $K$-equivariant index of a Dirac operator on $G/H$, for a Cartan subgroup $H<G$. Since $G/H$ is noncompact in general, the kernel of such an operator is infinite-dimensional. This is desirable, since $\pi$ is infinite-dimensional, but it makes the definition of the index more involved than on a compact manifold. We use index theory developed by Braverman \cite{Braverman02}. 

Our construction involves a map
\[
\Phi\colon G/H \to \kk^*
\]
of the form $\Phi(gH) = (\Ad^*(g)\xi)|_K$, for a fixed $\xi \in \kg^*$, and where $g \in G$. This is a moment map in the sense of symplectic geometry, although our construction will take us into the more general almost complex or $\Spinc$-setting. After we identify $\kk^* \cong \kk$ via an $\Ad(K)$-invariant inner product, the map $\Phi$ induces a vector field $v^{\Phi}$ via the infinitesimal action by $\kk$ on $G/H$. We define a $K$-invariant almost complex structure $J$ on $G/H$. This induces a Clifford action $c$ by $T(G/H)$ on $\Bigwedge_J T(G/H)$. Here $\Bigwedge_J$ stands for the exterior algebra of complex vector spaces. Let $D$ be a Dirac operator on $\Bigwedge_J T(G/H)$. We write $\Bigwedge_J^{\pm} T(G/H)$ for the even and odd degree parts of this bundle, respectively. As a special case of Theorem 2.9 in \cite{Braverman02}, we find that the multiplicities $m_{\delta}^{\pm}$ of irreducible representations $\delta$ of $K$ in 
\[
\ker(D - ifc(v^{\Phi})) \cap L^2(\Bigwedge_J^{\pm} T(G/H)\otimes L_{\pi})
\]
are finite, for a function $f$ with suitable growth behaviour. Here $L_{\pi} \to G/H$ is a certain line bundle associated to $\pi$. This deformation of the Dirac operator goes back to Tian--Zhang \cite{Zhang98}, who used it to give a proof of  Guillemin--Sternberg's quantisation commutes with reduction conjecture. Furthermore, $m_{\delta}^+ - m^-_{\delta}$ is independent of $f$ and of the specific Dirac operator $D$ used. This allows us to consider the $K$-equivariant index
\[
\indx_K(\Bigwedge_J T(G/H)\otimes L_{\pi}, \Phi) := \bigoplus_{\delta \in \hat K} (m_{\delta}^+ - m_{\delta}^-)\delta.
\]
This index defines an element of the completion $\Hom_{\Z}(R(K), \Z)$ of the representation ring $R(K)$ of $K$. It equals indices defined and used by Paradan--Vergne \cite{Paradan03, Paradan11, Vergne06} and Ma--Zhang \cite{Zhang14}.

The main result in this paper, Theorem \ref{main theorem}, states that this index equals $\pi|_K$, up to a sign.
\begin{theorem}
\label{thm main result intro}
We have
\[
\pi|_K = \pm \indx_K(\Bigwedge_J T(G/H)\otimes L_{\pi}, \Phi).
\]
\end{theorem}
See Section \ref{sec 3 def Dirac} for precise definitions of the sign $\pm$, and of $\Phi$, $J$ and $L_{\pi}$. In Subsection \ref{sec SL2}, we illustrate this result by working out what it means for $G = \SL(2,\R)$.

\subsection{Relation with geometric quantisation of coadjoint orbits}

If the infinitesimal character $\chi$ of $\pi$ is a regular element of $i\kh^*$, then Theorem \ref{thm main result intro} is a direct realisation of $\pi|_K$ as the $K$-equivariant geometric quantisation of a coadjoint orbit of $G$, as in the orbit method. Indeed, we may then take the map $\Phi$ to be the composition
\beq{eq Phi intro}
\Phi\colon G/H \xrightarrow{\cong} \Ad^*(G)(\chi+ \tilde \rho) \hookrightarrow \kg^* \to \kk^*,
\eeq
for an element $\tilde \rho \in i\kh^*$  defined in terms of half sums of positive roots. Then $\Phi$ is the natural moment map in the symplectic sense for the action by $K$ on the coadjoint orbit $\Ad^*(G)(\chi+ \tilde \rho)$. The line bundle $L_{\pi}$ is now such that   $\Bigwedge_J T(G/H)\otimes L_{\pi}$ is the spinor bundle of a $\Spinc$-structure with determinant line bundle
\[
G \times_H \C_{2(\chi + \tilde \rho)} \to G/H = \Ad^*(G)(\chi+ \tilde \rho),
\]
twisted by a one-dimensional representation of a finite Cartesian factor of $H$. In fact, $\Phi$ is a moment map in the $\Spinc$-sense \cite{PV14} for this $\Spinc$-structure. Therefore, Theorem \ref{thm main result intro} now states that $\pi|_K$ is the $K$-equivariant $\Spinc$-quantisation \cite{HS16, Paradan03, PV14} of $\Ad^*(G)(\chi + \tilde \rho)$.

Index theory of Dirac operators deformed by the vector field $v^{\Phi}$, for a moment map $\Phi$, appears frequently and naturally in geometric quantisation. (There are at least two other, but equivalent, definitions to the one we use here, used in \cite{Paradan01, Paradan11, Vergne06} and \cite{Zhang14}, respectively.) In the compact case, such deformations were used to prove {quantisation commutes with reduction} results \cite{Paradan01, PV14, Zhang98}. In the noncompact case, they are also used to define geometric quantisation \cite{ Mathai14, Mathai13,HS16, Zhang14, Paradan11, Vergne06}. So the use of deformed Dirac operators in Theorem \ref{thm main result intro} is natural from the point of view of geometric quantisation. A concrete consequence of this is that it allows us to apply the {quantisation commutes with reduction} principle
to compute the multiplicities of the $K$-types of $\pi$, as we do in \cite{PHYSSY2}.

Paradan  \cite{Paradan03} has pointed out  that $\Spinc$-quantisation is the relevant notion of geometric quantisation here; i.e.\ one should view $\Ad^*(G)(\chi + \tilde \rho)$ as a $\Spinc$-manifold rather than as a symplectic manifold. More specifically, the $\Spinc$-version of the quantisation commutes with reduction principle applies here, which we use in \cite{PHYSSY2} to deduce a geometric expression for the multiplicities of the $K$-types of $\pi$ from Theorem \ref{thm main result intro}. Paradan and Vergne \cite{PV14} showed that that principle has a natural generalisation from the symplectic setting to the $\Spinc$-setting. The version we use in \cite{PHYSSY2} is the result for noncompact $\Spinc$-manifolds proved in \cite{HS16}.

If $\chi$ is singular, then in the orbit method, $\pi$ is associated to a nilpotent orbit (which need not be the orbit through $\chi+ \tilde \rho$). In this case, the first map in \eqref{eq Phi intro} is a fibre bundle. By using $G/H$ rather than this nilpotent orbit, we do not directly deal with the problem of quantising nilpotent orbits, but we are able to use Theorem \ref{thm main result intro} to obtain a multiplicity formula for $K$-types for all tempered representations in \cite{PHYSSY2}.

\subsection{Ingredients of the proof}

There are several challenges in generalising Theorem \ref{thm main result intro} from discrete series representations 
 to arbitrary tempered representations.
\begin{enumerate}
\item
The space $G/H$ does not have a naturally defined $G$-invariant almost complex structure.
\item We do not have an explicit result like Blattner's formula to base the construction on.
\item If $T<K$ is a maximal torus and $\kg = \kk \oplus \ks$ is a Cartan decomposition, then in the discrete series case, we have a $K$-equivariant diffeomorphism $G/T = K\times_T \ks$. Such a ``partial linearisation" is more complicated in the general case.
\end{enumerate}
The last of these points takes most work to solve. 

To deduce a multiplicity formula for $K$-types from Theorem \ref{thm main result intro}, there are two main challenges.
\begin{enumerate}
\item Paradan showed in \cite{Paradan03} that one needs a version of the quantisation commutes with reduction principle for noncompact $\Spinc$-manifolds. He proved such a result in the setting relevant to discrete series representations, but one needs a more general version for arbitrary tempered representations.
\item It is unclear what coadjoint orbits, or what maps $\Phi$, one should use in general, for example for limits of the discrete series.
\end{enumerate}
The first of these points was solved in \cite{HS16}. There a general version of the quantisation commutes with reduction principle was proved  for noncompact $\Spinc$-manifolds. This was based on Paradan--Vergne's result for compact $\Spinc$-manifolds in \cite{Paradan14a, PV14}. The result in   \cite{HS16} is an analogue of the result by Ma--Zhang \cite{Zhang14} for noncompact symplectic manifolds in the more general $\Spinc$-setting. The second point will be solved in   \cite{PHYSSY2}.

The proof of Theorem \ref{thm main result intro} in this paper consists of 3 steps.
\begin{enumerate}
\item Prove that the right hand side of the equality in Theorem \ref{thm main result intro} equals an index on a ``partially linearised" space $E$ that is $K$-equivariantly diffeomorphic to $G/H$. This is done in Section \ref{sec linearise}; see Proposition \ref{prop linearise}.
\item Compute this index on $E$ explicitly. This is done in Sections \ref{sec ind fibred} and \ref{sec linearised index}; see Proposition \ref{prop lin index}.
\item Use that explicit expression to prove that the index on $E$ equals $\pi|_K$. This is done in Section \ref{sec pi K}; see Proposition \ref{prop lin comp}.
\end{enumerate}
The second step is the most elaborate. One reason for this is that the arguments we use involve deformations of Dirac operators that do not fit into the index theory developed by Braverman. That means that we have to use homotopy arguments specifically tailored to our situation, rather than the general cobordism invariance property of Braverman's index.

The representation theoretic input to our proof of Theorem \ref{thm main result intro} is:
\begin{itemize}
\item the part of Knapp and Zuckerman's classification of tempered representation that states that every tempered representation is basic (Corollary 8.8 in \cite{KZ1});
\item Blattner's formula for multiplicities of $K$-types of (limits of) discrete series representations \cite{HS75}.
\end{itemize}
In fact, to be precise, Theorem \ref{thm main result intro} applies to all basic (or standard) representations $\pi$ (see Remark \ref{rem basic}). The first of the above two ingredients is not necessary for the proof of the result in that formulation. 

\subsection*{Acknowledgements}

The authors are grateful to Maxim Braverman, Paul-\'Emile Paradan and David Vogan for their hospitality and inspiring discussions at various stages.

The first author was partially supported by the European Union, through Marie Curie fellowship PIOF-GA-2011-299300. He thanks Dartmouth College for funding a visit there in 2016. The third author was supported by the Direct Grants and Research Fellowship Scheme from the Chinese University of Hong Kong. 

\subsection*{Notation}

The Lie algebra of a Lie group is denoted by the corresponding lower case Gothic letter. We denote complexifications by superscripts $\C$.  The unitary dual of a group $H$ will be denoted by $\hat H$. If $H$ is an abelian Lie group and $\xi \in \kh^*$ satisfies the appropriate integrality condition, then we write $\C_{\xi}$ for the one-dimensional representation of $H$ with weight $\xi$. 

In Subsections
\ref{sec deformed Dirac} and \ref{sec prop index} and in Section
\ref{sec ind fibred}, the letter $M$ will denote a manifold. In the rest of this paper, $M$ is a subgroup of the group $G$. In Section \ref{sec ind fibred}, $N$ is another manifold, whereas $N$ denotes a subgroup of $G$ in the rest of this paper. (There is little risk of confusion, because the group $G$ does not play a role at all in the sections where $M$ and $N$ are manifolds.) The Levi--Civita connection on a Riemannian manifold $M$ will be denoted by $\nabla^{TM}$.


\section{Tempered representations}\label{sec tempered}

We start by reviewing the basic properties of tempered representations that we will need, including a part of their classification by Knapp and Zuckerman \cite{KZ1, KZ2, KZ3}.

Throughout this paper, except in Subsection \ref{sec limit ds},
$G$ will denote a connected, linear, real reductive Lie group with compact centre $Z_G$. (This is the class of groups  for which tempered representations were classified in \cite{KZ1, KZ2, KZ3}.) 
Let $K<G$ be a maximal compact subgroup. Let $\theta$ be the corresponding Cartan involution, with Cartan decomposition $\kg = \kk \oplus \ks$. Let $(\relbar, \relbar)$ be the $K$-invariant inner product on $\kg$ defined by the Killing form and $\theta$. We transfer this inner product to the dual spaces $\kg^*$ and $i\kg^*$ where necessary.

A unitary irreducible representation $\pi \in \hat G$  is \emph{tempered} if all of its $K$-finite matrix coefficients belong to $L^{2_+ \varepsilon}(G)$ for all $\varepsilon > 0$. If $\hat G_{\temp}$ is the set of tempered representations of $G$, then 
we have the Plancherel decomposition
\beq{eq Plancherel}
L^2(G) = \int^{\oplus}_{\hat G_{\temp}} \pi \otimes \pi^*\, d\mu(\pi),
\eeq
as representations of $G \times G$, where $\mu$ is the Plancherel measure. Tempered representations are important 
\begin{enumerate}
\item to harmonic analysis, because of \eqref{eq Plancherel};
\item because they are used in the Langlands classification of all admissible representations \cite{Langlands89}; see also e.g.\ Section VIII.15 in \cite{KnappBook}.
\end{enumerate}

\subsection{Limits of discrete series}\label{sec limit ds}

In this subsection and the next, we make different assumptions on $G$ than in the rest of this paper. (The contents of these subsections will later be applied to the subgroup $M<G$ in the Langlands decomposition $P = MAN$ of a cuspidal parabolic subgroup $P<G$.) Suppose $G$ is a real linear Lie group, not necessarily connected. Let $G_0 <G$ be the connected component of the identity element. We now assume that 
\begin{enumerate}
\item $\kg$ is reductive;
\item $G_0$ has compact centre;
\item $G$ has finitely many connected components;
\item if $G^{\C}$ is the analytic linear Lie group with Lie algebra $\kg^{\C}$, and if $Z(G)$ is the centraliser  of $G$ in the full general linear matrix group containing $G$, then $G \subset G^{\C}Z(G)$.
\end{enumerate}
These assumptions imply those used by Harish--Chandra in \cite{HC75}, see Section 1 of \cite{Knapp82}.

In addition to the above assumptions, we suppose that $G$ has a compact Cartan subgroup $T<K$. Then it has  discrete series and limits of discrete series representations. We recall the classification of those representations, taking into account the fact that $G$ may be disconnected. We refer to Section 1 of \cite{KZ1} for details. See also Sections IX.7 and XII.7 in \cite{KnappBook} for the connected case.

Let $R_G = R(\kg^{\C}, \kt^{\C})$ be the root system of $(\kg^{\C}, \kt^{\C})$.
For a regular element $\lambda \in i\kt^*$, let $\rho_{\lambda}$ be half the sum of the elements of $R_G$ with positive inner products with $\lambda$. Then the discrete series of $G_0$ is parametrised by the set of regular elements of  $\lambda \in i\kt^*$ for which $\lambda - \rho_{\lambda}$ is integral (i.e.\ lifts to a homomorphsm $e^{\lambda - \rho_{\lambda}}\colon T \to \U(1)$). For such an element $\lambda$, let $\pi^{G_0}_{\lambda}$ be the corresponding discrete series representation. For another such element $\lambda'$, we have $\pi^{G_0}_{\lambda} \cong \pi^{G_0}_{\lambda'}$ if and only if there is an element $w$ of the Weyl group $N_{G_0}(T)/Z_{G_0}(T)$ such that $\lambda' = w\lambda$.

Let $\lambda \in i\kt^*$ be as above. 
Let $\chi \in \hat Z_G$ be such that
\begin{equation}\label{eq chi}
\chi|_{T \cap Z_G} = e^{\lambda - \rho_{\lambda}}|_{T \cap Z_G} .
\end{equation}
Then we have the well-defined representation $\pi^{G_0}_{\lambda} \boxtimes \chi$ of $G_0 Z_G$, given by
\[
(\pi^{G_0}_{\lambda} \boxtimes \chi) (gz) =   \pi^{G_0}_{\lambda} (g) \chi(z),
\]
for $g \in G_0$ and $z \in Z_G$. Write
\[
\pi^G_{\lambda, \chi} := \Ind_{G_0 Z_G}^G(\pi^{G_0}_{\lambda} \boxtimes \chi)
\]
(Here and in the rest of this paper, $\Ind$ denotes normalised induction.)
This is a discrete series representation of $G$, and
all discrete series representations of $G$ are of this form. Two such representations $\pi^G_{\lambda, \chi}$ and $\pi^G_{\lambda', \chi'} $ are equivalent if and only if $\chi' = \chi$ and there is an element $w \in N_{G}(T)/Z_{G}(T)$ such that $\lambda' = w\lambda$.

Now let $\lambda \in i\kt^*$ be possibly singular, and choose a (non-unique) positive root system $R^+_G \subset R_G$ of roots having nonnegative inner products with $\lambda$. Let $\rho$ be half the sum of the elements of $R^+_G$, and suppose that $\lambda - \rho$ is integral.

 Let $\nu \in i\kt^*$ be regular and dominant with respect to $R^+_G$. Let $F_{-\nu}$ be the finite-dimensional irreducible representation of $\kg^{\C}$ with lowest weight $-\nu$. Set
\beq{eq lim ds}
\pi^{G_0}_{\lambda, R^+_G} := p_{\lambda}(\pi^{G_0}_{\lambda + \nu} \otimes F_{-\nu}),
\eeq
where $p_{\lambda}$ denotes projection onto the subspace with infinitesimal character $\lambda$. 
 If $\lambda$ is regular, then $\pi^{G_0}_{\lambda, R^+_G} = \pi^{G_0}_{\lambda}$. And even if $\lambda$ is singular, the right hand side of \eqref{eq lim ds} is independent of $\nu$. For singular $\lambda$, the representation $\pi^{G_0}_{\lambda, R^+_G}$ is a limit of discrete series representation of $G_0$.

Let $\lambda$ and $R^+_G$ be as above. Let $\chi$ be a unitary irreducible representation of $Z_G$, such that \eqref{eq chi} holds.
Write
\beq{eq lds G}
\pi^G_{\lambda, R^+_G, \chi} := \Ind_{G_0Z_G}^G(\pi^{G_0}_{\lambda, R^+_G} \boxtimes \chi).
\eeq
The following result is Theorem 1.1 in \cite{KZ1}.
\begin{theorem}
For $\lambda$, $R^+_G$ and $\chi$ as above, the representation $\pi^G_{\lambda, R^+_G, \chi}$ is
\begin{itemize}
\item nonzero if and only if $(\lambda, \alpha) \not= 0$ for all simple (with respect to $R^+_G$) compact roots $\alpha$;
\item  irreducible and tempered in that case.
\end{itemize}
If two such  representations  $\pi^G_{\lambda, R^+_G, \chi}$ and $\pi^G_{\lambda', (R^+_G)', \chi'}$ are nonzero, they are equivalent if and only if $\chi' = \chi$, and there is an element $w \in N_G(T)/Z_G(T)$ such that $\lambda' = w\lambda$ and $(R^+_G)' = wR^+_G$.
\end{theorem}
The limits of discrete series representations are the nonzero representations occurring in the above theorem.

%
%

\subsection{The Knapp--Zuckerman classification} \label{sec KZ}

We now return to the setting described at the start of this section. In particular, $G$ is connected, linear, real reductive, with compact centre. We state the part that we need of  Knapp and Zuckerman's classification of tempered representations of $G$, in terms of limits of discrete series representations of subgroups of $G$. At the same time, we fix notation that will be used in the rest of this paper.

Let $\kh \subset \kg$ be a $\theta$-stable Cartan subalgebra. Set
\begin{itemize}
\item $H := Z_G(\kh)$;
\item $\ka := \kh \cap \ks$;
\item $A :=$ the analytic subgroup of $G$ with Lie algebra $\ka$;
\item $\km :=$ the orthogonal complement to $\ka$ in $Z_{\kg}(\ka)$;
\item $M_0 :=$ the analytic subgroup of $G$ with Lie algebra $\km$;
\item $M := Z_K(\ka)M_0$.
\end{itemize}
The subgroup $M$ may be disconnected. But importantly, it satisfies the assumptions made on the group $G$ in Subsection \ref{sec limit ds}.

For $\beta \in \ka^*$, set
\[
\kg_{\beta} := \{X \in \kg; \text{for all $Y \in \ka$, $[Y,X] = \langle \beta, Y\rangle X$}\}.
\]
Consider the restricted root system
\[
\Sigma := \Sigma(\kg, \ka) := \{\beta \in \ka^* \setminus \{0\}; \kg_{\beta} \not= \{0\} \}.
\]
Fix a positive system $\Sigma^+ \subset \Sigma$. Consider the nilpotent subalgebras
\[
\kn^{\pm} := \bigoplus_{\beta \in \Sigma^+} \kg_{\pm \beta}.
\]
of $\kg$.
We will write $\kn := \kn^+$. Let $N$ be the analytic subgroup of $G$ with Lie algebra $\kn$. Then $P := MAN$ is a parabolic subgroup of $G$.

Let $T<K$ be a maximal torus. Set
\begin{itemize}
\item $K_M := K \cap M$;
\item $\kt_{M} := \kk_M \cap \kt$;
\item $T_M := \exp(\kt_M)$.
\end{itemize}
Then $\kt_M \subset \km$ is a Cartan subalgebra, so $M$ has discrete series and limits of discrete series representations. That is to say, $P$ is a cuspidal parabolic subgroup. In fact, all cuspidal parabolic subgroups occur in this way. 

If we write $\ks_{M} := \km \cap \ks$, then we obtain the Cartan decomposition $\km = \kk_{M} \oplus \ks_{M}$. Set $H_M := H \cap M$. We have $T_M < H_M$. The converse inclusion does not hold in general, since $T_M$ is connected, whereas $H_M$ may be disconnected. More explicitly, Corollary 7.111 in \cite{Knapp02} implies that
\beq{eq HM}
H_M = T_MZ_M.
\eeq
In fact, $H_M = T_MZ_M'$ for a finite subgroup $Z_M' < Z_M$. So
 the Lie algebra of $H_M$ is $\kt_M$.

Let $\lambda \in i\kt_M^*$, $R^+_M \subset R(\km^{\C}, \kt_M^{\C})$, and $\chi_M \in \hat Z_M$ be as in Subsection \ref{sec limit ds}, with $G$ replaced by $M$ and $T$ by $H_M$. Then we have the limit of discrete series representation $\pi^{M}_{\lambda, R^+_M, \chi_M}$ of $M$. Let $\nu \in i\ka^*$. A \emph{basic representation} of $G$ is a representation of the form
\[
\Ind_P^G(\pi^{M}_{\lambda, R^+_M, \chi_M} \otimes e^{\nu} \otimes 1_N),
\]
where $1_N$ is the trivial representation of $N$.

\begin{theorem}[Knapp--Zuckerman]\label{thm KZ}
Every tempered representation of $G$ is basic. 
\end{theorem}
This is Corollary 8.8 in \cite{KZ1}. In Theorem 14.2 in \cite{KZ2}, Knapp and Zuckerman complete the classification of tempered representations by showing  which basic representations are irreducible and tempered. (These are the ones with nondegenerate data and trivial $R$-groups; see Sections 8 and 12 in \cite{KZ2} for details on these conditions). 

The main result of this paper, Theorem \ref{main theorem}, is formulated for tempered representations, but in fact applies more generally to all basic representations (see Remark \ref{rem basic}). The result is formulated for tempered representations, because of the special relevance of those representations.


\section{Indices of deformed Dirac operators} \label{sec 3 def Dirac}

The main result in this paper is Theorem \ref{main theorem}, which states that the restriction to $K$ of a tempered representation of $G$ can be realised as the equivariant index of a deformed Dirac operator on $G/H$, for a Cartan subgroup $H<G$. In this section, we review the index theory we will use, and state the main result.

\subsection{Deformed Dirac operators}\label{sec deformed Dirac}

Braverman \cite{Braverman02} developed equivariant index theory for the deformations of Dirac operators on noncompact manifolds that we briefly discuss in this subsection. His index is the same as the indices defined by Paradan and Vergne \cite{Paradan11, Vergne06} (see Theorem 5.5 in \cite{Braverman02}) and Ma--Zhang \cite{Zhang09} (see Theorem 1.5 in \cite{Zhang09}). Its main applications have so far been to geometric quantisation and representation theory, see e.g.\ \cite{Paradan03}. The deformation of Dirac operators \eqref{eq deformed Dirac} used by Braverman was introduced by Tian and Zhang \cite{Zhang98} in their analytic proof of Guillemin and Sternberg's \emph{quantisation commutes with reduction} problem (which was first proved by Meinrenken \cite{Meinrenken98} and Meinrenken--Sjamaar \cite{Meinrenken99}; another proof was given by Paradan \cite{Paradan01}).

In this subsection, we consider a complete Riemannian manifold $M$, on which a compact Lie group $K$ acts isometrically. (In this subsection and the next, $M$ does not denote a subgroup of $G$; in fact $G$ does not play a role at all here.) Let $\cS \to M$ be a $\Z_2$-graded, Hermitian, $K$-equivariant complex vector bundle. Let $c \colon TM \to \End(\cS)$ be a $K$-equivariant vector bundle homomorphism, called the \emph{Clifford action}, such that for all $v \in TM$,
\[
c(v)^2 = -\|v\|^2.
\]
(Here $K$ acts on $\End(\cS)$ by conjugation.) Then $\cS$ is called a $K$-equivariant \emph{Clifford module}.
\begin{example} \label{ex S J}
In the setting we consider in the rest of this paper, $M$ will have a $K$-equivariant almost complex structure $J$, and we will use  $\cS = \Bigwedge_J TM \otimes L$, where $L\to M$ is a line bundle and  $ \Bigwedge_J TM$ is the complex exterior algebra bundle of $TM$ with respect to $J$. 
This has a natural Clifford action, given by
\[
c(v)x = v\wedge x - v^* \lrcorner x,
\]
where $v \in T_mM$ for some $m \in M$, $x \in \Bigwedge_J T_mM \otimes L_m$, $v^* \in T^*_mM$ is dual to $v$ with respect to the Hermitian metric defined by $J$ and the Riemannian metric, and $\lrcorner$ denotes contraction (see e.g.\ page 395 in \cite{Lawson89}). When dealing with vector bundles of the form $\Bigwedge_J TM \otimes L$, we will always use this Clifford action.
\end{example}
Let $\nabla$ be a $K$-invariant, Hermitian connection on $\cS$ such that for all vector fields $v$ and $w$ on $M$,
\[
[\nabla_v, c(w)] = c(\nabla^{TM}_v w),
\]
where $\nabla^{TM}$ is the Levi--Civita connection on $TM$. If we identify $T^*M \cong TM$ via the Riemannian metric, we can view $c$ as a vector bundle homomorphism
\[
c\colon T^*M \otimes \cS \to \cS.
\]
The \emph{Dirac operator} $D$ associated to $\nabla$ is the composition
\[
D\colon \Gamma^{\infty}(\cS) \xrightarrow{\nabla} \Gamma^{\infty}(T^*M \otimes \cS) \xrightarrow{c} \Gamma^{\infty}(\cS).
\]
It is odd with respect to the grading on $\cS$; we denote its restrictions to even and odd sections by $D^+$ and $D^-$, respectively.

If $M$ is compact, then the elliptic operator $D$ has finite-dimensional kernel. So it has a well-defined equivariant index
\[
\indx_K(D) := [\ker(D^+)] - [\ker(D^-)] \quad \in R(K),
\]
where $R(K)$ is the representation ring of $K$, and square brackets denote equivalence classes of representations of $K$. If $M$ is noncompact, one can still define an equivariant index, using a \emph{taming map}.

Let $\psi\colon M \to \kk$ be an equivariant smooth map (with respect to the adjoint action by $K$ on $\kk$). Let $v^{\psi}$ be the vector field on $M$ defined by
\[
v^{\psi}(m) = \ddt \exp(-t\psi(m))\cdot m,
\]
for $m \in M$. The map $\psi$ is called a taming map if the set of zeroes of the vector field $v^{\psi}$ is compact.
The \emph{Dirac operator deformed by $\psi$} is the operator
\beq{eq deformed Dirac}
D_{\psi} := D - ic(v^{\psi})
\eeq
on $\Gamma^{\infty}(\cS)$. As for the undeformed operator, we denote the restrictions of $D_{\psi}$ to even and odd sections by $D^+_{\psi}$
 and $D^-_{\psi}$, respectively.
 
To obtain a well-defined index, one needs to rescale the map  $\psi$ by a function with suitable growth behaviour. Let 
$
\Phi^{\cS} \in \End(\cS) \otimes \kk^*
$
be given by
\beq{eq def mu S}
\langle \Phi^{\cS}, Z\rangle = \nabla_{Z^M} - \calL_{Z},
\eeq
for $Z \in \kk$. Here $Z^M$ is the vector field on $M$ induced by $Z$ via the infinitesimal action.  Let $\nabla^{TM}$ be the Levi--Civita connection on $TM$. Consider the positive, $K$-invariant function
\beq{eq def h}
h := \|v^{\psi}\| + \|\nabla^{TM}v^{\psi}\| +\|\langle \Phi^{\cS}, \psi\rangle\| + \|\psi\| + 1
\eeq
on $M$. A nonnegative function $f \in C^{\infty}(M)^K$ is said to be \emph{admissible} (for $\psi$ and $\nabla$) if
\[
\frac{f^2\|v^{\psi}\|^2}{\|df\| \|v^{\psi}\| + fh + 1}(m) \to \infty
\]
as $m \to \infty$ in $M$. It is shown in 
 Lemma 2.7 in \cite{Braverman02} that admissible functions always exist.
\begin{theorem}[Braverman]\label{thm Braverman}
Suppose $\psi$ is taming. Then for all admissible functions $f \in C^{\infty}(M)^K$, and all $\delta \in \hat K$, the multiplicity $m^{\pm}_{\delta}$ of $\delta$ in 
\[
\ker(D^{\pm}_{f\psi}) \cap L^2(\cS)
\]
is finite. The difference $m^+_{\delta} - m^{-}_{\delta}$ is independent of $f$ and $\nabla$.
\end{theorem}
This is Theorem {2.9} in \cite{Braverman02}. The fact that $m^+_{\delta} - m^{-}_{\delta}$ is independent of $f$ and $\nabla$ is a consequence of Braverman's cobordism invariance result, Theorem 3.7 in \cite{Braverman02}. This cobordism invariance property also implies that the index is independent of the $K$-invariant, complete Riemannian metric on $M$. 

Let $\hat R(K)$ be the abelian group
\[
\hat R(K) := \Bigl\{\sum_{\delta \in \hat K}m_{\delta} \delta; m_{\delta} \in \Z \Bigr\}.
\]
It contains the representation ring $R(K)$ as the subgroup for which only finitely many of the coefficients $m_{\delta}$ are nonzero.
\begin{definition}
In the setting of Theorem \ref{thm Braverman}, the \emph{equivariant index} of the pair $(\cS, \psi)$ is
\[
\indx_K(\cS,\psi) := \sum_{\delta \in \hat K} (m^{+}_{\delta} - m^-_{\delta}) \delta \quad \in \hat R(K).
\]
\end{definition}
This index was generalised to proper actions by noncompact groups in Theorem 3.12 in \cite{HochsSong15-2}. Earlier, this was done for sections invariant under the group action in \cite{Braverman14, Mathai13}.

\subsection{Properties of the index}\label{sec prop index}

As mentioned above, independence of the coefficients  $m^+_{\delta} - m^{-}_{\delta}$ of the choices of $f$ and $\nabla$ follows from a cobordism invariance property. We will use the special case of this result that we describe now. For $j = 1,2$, let  $\cS_j \to M$ be a Clifford module, and let $\psi_j\colon M \to \kk$ be a taming map.
\begin{definition}\label{def htp}
A \emph{homotopy} between $(\cS_1, \psi_1)$ and $(\cS_2, \psi_2)$ is a pair $(\cS, \psi)$, where 
\begin{itemize}
\item $\cS \to M\times [0,1]$ is a Clifford module, such that 
\begin{itemize}
\item[$\circ$] $\cS|_{M \times [0,1/3[} = \cS_1 \times [0,1/3[$, including the Clifford actions by $TM$;
\item[$\circ$] $\cS|_{M \times ]2/3,1]} = \cS_2 \times ]2/3, 1]$, including the Clifford actions by $TM$.
\end{itemize}
Let $\partial_t$ be the unit vector field of the component $\R$ in $T(M\times {]0,1[}) = TM \times \R \times {]0,1[}$. Let $c\colon T(M\times {]0,1[})\to \End(\cS|_{M\times {]0,1[}})$ be the Clifford action. Then
\begin{itemize}
\item[$\circ$] $c(\partial_t)|_{M \times [0,1/3[} = \ii$;
\item[$\circ$] $c(\partial_t)|_{M \times ]2/3, 1[} = -\ii$.
\end{itemize}
\item $\psi \colon M \times [0,1] \to \kk$ is a taming map, such that for all $m \in M$ and $t \in [0,1]$, 
\[
\psi(m, t) = \left\{ \begin{array}{ll}
\psi_1(m) & \text{if $t<1/3$;} \\
\psi_2(m) & \text{if $t >2/3$.}
 \end{array}\right.
\]
 \end{itemize}
\end{definition}
\begin{theorem}[Homotopy invariance]\label{thm htp}
If $(\cS_1, \psi_1)$ and $(\cS_2, \psi_2)$ are homotopic, then
\[
\indx_K(\cS_1, \psi_1) = \indx_K(\cS_2, \psi_2).
\]
\end{theorem}
See Theorem 3.7 in \cite{Braverman02}. We will only apply this result in cases where either $\cS_1 = \cS_2$ or $\psi_1 = \psi_2$.  If $(\cS_1, \psi_1)$ and $(\cS_2, \psi_2)$ are homotopic, then if $\cS_1 = \cS_2$, we say that $\psi_1$ and $\psi_2$ are homotopic. Similarly, if $\psi_1 = \psi_2$, then  we say that $\cS_1$ and $\cS_2$ are homotopic.

\begin{corollary} \label{cor pos inner prod}
Suppose that $\cS_1 = \cS_2$, and that
\[
(v^{\psi_1}, v^{\psi_2}) \geq 0.
\]
Then $\psi_1$ and $\psi_2$ are homotopic, so that
\[
\indx_K(\cS_1, \psi_1) = \indx_K(\cS_2, \psi_2).
\]
\end{corollary}
\begin{proof}
Let $\chi\colon \R \to [0,1]$ be a smooth function such that $\chi(t) = 0$ if $t<1/3$ and $\chi(t) = 1$ if $t>2/3$. For $t \in [0,1]$ and $m \in M$, set
\[
\psi(m,t) := (1-\chi(t))\psi_1(m) + \chi(t)\psi_2(m).
\]
Then
\begin{multline*}
\|v^{\psi_t}(m)\|^2 = (1-\chi(t))^2 \|v^{\psi_1}(m)\|^2 + \chi(t)^2 \|v^{\psi_2}(m)\|^2 + 2(1-\chi(t))\chi(t) (v^{\psi_1}, v^{\psi_2}) \\
 \geq  (1-\chi(t))^2 \|v^{\psi_1}(m)\|^2 + \chi(t)^2 \|v^{\psi_2}(m)\|^2.
\end{multline*}
This can only vanish if $v^{\psi_1}(m) = 0$ or $v^{\psi_2}(m) = 0$, so that $\psi$ is a taming map. Hence the claim follows from Theorem \ref{thm htp}. 
\end{proof}

Another property of the index that we will use is excision. 
\begin{proposition}[Excision]\label{prop excision}
Suppose that $\psi$ is taming, and let $U \subset M$ be a relatively compact, $K$-invariant open subset, with a smooth boundary, outside which $v^{\psi}$ does not vanish. Consider a $K$-invariant Riemannian metric on $TU$ for which $U$ is complete, and which coincides with the Riemannian metric on $TM$ in a neighbourhood of the zeroes of $v^{\psi}$. Also consider a compatible Clifford action by  $TU$ on $\cS|_U$. Then
\[
\indx_K(\cS|_U, \psi|_U) = \indx_K(\cS, \psi).
\]
\end{proposition}
For a proof, see Lemma 3.12 and Corollary 4.7 in \cite{Braverman02}.

\subsection{The discrete series case}

We return to the setting described at the start of Section \ref{sec tempered}.
One application of the index theory of deformed Dirac operators was Paradan's realisation in \cite{Paradan03} of restrictions to $K$ of discrete series representations of $G$. The main result in this paper, Theorem \ref{main theorem}, is a generalisation of Paradan's result to arbitrary tempered representations. The main advantage of Paradan' result is that it implies a geometric formula for multiplicities of $K$-types of discrete series representations, Theorem 1.5 in \cite{Paradan03}. We will use Theorem \ref{main theorem} to obtain a generalisation of this multiplicity formula to arbitrary tempered representations in a forthcoming paper  \cite{PHYSSY2}.

Let us first state Paradan's result. Suppose $G$ is semisimple, and has a compact Cartan subgroup $T<K$. Let $\lambda \in i\kt^*$ be a regular element for which $\lambda - \rho_{\lambda}$ is integral, and let $\pi^G_{\lambda}$ be the corresponding discrete series representation. The coadjoint orbit
\[
\Ad^*(G) \lambda \cong G/T
\]
has a $G$-invariant complex structure $J$ such that
\beq{eq complex gt}
T_{eT}(G/T) = \kg/\kt = \bigoplus_{\alpha \in R^+_G} \kg^{\C}_{\alpha}
\eeq
as complex vector spaces. Let 
\[
 \Bigwedge_{J} T(G/T) \to G/T
\]
be the corresponding Clifford module, as in Example \ref{ex S J}.

Let $\C_{\lambda - \rho_{\lambda}}$ be the one-dimensional representation of $T$ with weight $\lambda - \rho_{\lambda}$, and consider the line bundle
\[
L_{\lambda - \rho_{\lambda}} := G\times_{T} \C_{\lambda - \rho_{\lambda}} \to G/T.
\]
Consider map
\[
\Phi\colon G/T = \Ad^*(G)\lambda \hookrightarrow \kg^* \to \kk^* \cong \kk,
\]
where the map $\kg^* \to \kk^*$ is the restriction map, and the identification $\kk^* \cong \kk$ is made via the $K$-invariant inner product chosen earlier. This map is taming by Proposition 2.1 in \cite{Paradan99}.
\begin{theorem}[Paradan]\label{thm Paradan}
We have
\beq{eq Paradan}
\pi^G_{\lambda}|_K = (-1)^{\dim(G/K)/2}
\indx( \Bigwedge_{J} T(G/T) \otimes L_{\lambda - \rho_{\lambda}}, \Phi).
\eeq
\end{theorem}
This is Theorem 5.1 in \cite{Paradan03}. In fact, Paradan proves that the right hand side of \eqref{eq Paradan} equals $(-1)^{\dim(G/K)/2}$ times the right hand side of Blattner's formula. The latter theorem therefore implies Theorem \ref{thm Paradan}.

\subsection{An almost complex structure} \label{sec ac str}

Now we drop the assumption that $G$ is semisimple and has a compact Cartan subgroup. We consider a general tempered representation $\pi$ of $G$, and write
\[
\pi = \Ind_P^G(\pi^{M}_{\lambda, R^+_M, \chi_M} \otimes e^{\nu} \otimes 1_N)
\]
as in Theorem \ref{thm KZ}. Let $H<G$ be the Cartan subgroup as in Subsection \ref{sec KZ}; we also use the other notation from that subsection. We will realise the restriction $\pi|_K$ as the $K$-equivariant index of an deformed Dirac operator on $G/H$ (up to a sign). To do this, we will use a $K$-equivariant almost complex structure on $G/H$. (If $\pi$ belongs to the discrete series, this is the $G$-invariant complex structure defined by \eqref{eq complex gt}, but in general it is only $K$-invariant, and need not be integrable.) We have
\[
\kg = \kk \oplus \ks_M \oplus \ka \oplus \kn.
\]
Therefore, we have an $H_M$-invariant decomposition
\beq{eq decomp gh 1}
\kg/\kh = \kk/\kt_M \oplus \ks_M \oplus \kn = \kk/\kk_M \oplus \kk_M/\kt_M \oplus \ks_M \oplus \kn.
\eeq
The map from $\kn^-$ to $\kk/\kk_M$, given by $X \mapsto \frac{1}{2}(X+\theta X)$ is an $H_M$-equivariant linear isomorphism. Using this, we find that, as a representation space of $H_M$, 
\beq{eq decomp gh}
\kg/\kh = \km/\kt_M \oplus \kn^- \oplus \kn^+.
\eeq
As in the discrete series case, the positive root system $R^+_M$ for $(\km^{\C}, \kt_M^{\C})$ determines an $\Ad(H_M)$-invariant complex structure $J_{\km/\kt_M}$ on $\km/\kt_M$ such that, as complex vector spaces,
\[
\km/\kt_M = \bigoplus_{\alpha \in R^+_M} \km^{\C}_{\alpha}.
\]

Once and for all, we fix an element $\zeta \in \ka$ such that $\langle \beta, \zeta \rangle > 0$ for all $\beta \in \Sigma^+$. Then the map
\[
\ad(\zeta) \colon\kn^- \oplus \kn^+ \to \kn^- \oplus \kn^+
\]
is invertible, with real eigenvalues. Set
\[
J_{\zeta} := \theta |\ad(\zeta)|^{-1} \ad(\zeta)\colon \kn^- \oplus \kn^+ \to \kn^- \oplus \kn^+.
\]
\begin{lemma}
The map $J_{\zeta}$ is an $H_M$-invariant complex structure.
\end{lemma}
\begin{proof}
The adjoint action by $H_M$ commutes with $\theta$ because $H_M < K$. It commutes with $\ad(\zeta)$, because $H_M <M < Z_G(\ka)$.
So the map $J_{\zeta}$ is $H_M$-equivariant. Let $\beta \in \Sigma^+$. The map $\ad(\zeta)$ preserves the spaces $\kg_{\pm \beta}$, while the Cartan involution $\theta$ interchanges them. Hence, if $X_{\beta} \in \kg_{\beta}$,
\[
J_{\zeta}^2(X_{\beta}) = J_{\zeta}(\theta X_{\beta}) = -X_{\beta}.
\]
Similarly, if $X_{-\beta} \in \kg_{-\beta}$, then
\[
J_{\zeta}^2(X_{-\beta}) = -J_{\zeta}(\theta X_{-\beta}) = -X_{-\beta}.
\]
\end{proof}

Let 
\beq{eq def Jgh}
J_{\kg/\kh} :=  J_{\km/\kt_M} \oplus J_{\zeta}
\eeq
be the $H_M$-invariant complex structure on $\kg/\kh$ defined by $J_{\km/\kt_M}$ and $J_{\zeta}$ via the isomorphism \eqref{eq decomp gh}. 

Since $J_{\kg/\kh}$ is not necessarily $H$-invariant, it does not extend to a $G$-invariant almost complex structure on $G/H$ in general. However, it does extend to a $K$-invariant almost complex structure in a natural way.
\begin{lemma}\label{lem J}
There is a unique $K$-invariant almost complex structure $J$ on $G/H$, such that for all $k \in K$, $X \in \ks_M$ and $Y \in \kn$, the following diagram commutes:
\beq{eq diag J}
\xymatrix{
T_{k\exp(X)\exp(Y)H}G/H \ar[r]^-{J} & T_{k\exp(X)\exp(Y)H}G/H  \\
T_{eH}G/H = \kg/\kh \ar[u]^-{T_{eH}k\exp(X)\exp(Y)} \ar[r]_-{J_{\kg/\kh}} & \kg/\kh  = T_{eH}G/H \ar[u]_-{T_{eH}k\exp(X)\exp(Y)}.
}
\eeq
\end{lemma}
\begin{proof}
If $k, k' \in K$, $X, X' \in \ks_M$ and $Y, Y' \in \kn$ such that
\[
k\exp(X)\exp(Y)H = k'\exp(X')\exp(Y')H,
\]
then there are $h \in H_M$ and $a \in A$ such that
\[
k'\exp(X')\exp(Y') = k\exp(X)\exp(Y)ha = kh \exp(\Ad(h^{-1})X)\exp(\Ad(h^{-1})Y)a.
\]
Since the multiplication map 
\beq{eq mult map}
K\times \exp(\ks_M) \times N \times A \to G
\eeq
is injective (indeed, a diffeomorphism), we must have $a = e$. So
\[
T_{eH} k'\exp(X')\exp(Y') = T_{eH} k\exp(X)\exp(Y) \circ T_{eH}h.
\]
Since $T_{eH}h\colon \kg/\kh \to \kg/\kh$ is induced by $\Ad(h)$, it commutes with $J_{\kg/\kh}$. Since \eqref{eq mult map} is a diffeomorphism, every element of $G/H$ is of the form $k\exp(X)\exp(Y)H$ for some $k \in K$, $X \in \ks_M$ and $Y \in \kn$. Hence the map $J$ is well-defined by commutativity of \eqref{eq diag J}. This defining property directly implies that $J$ is $K$-invariant.
\end{proof}

\subsection{Tempered representations; the main result} \label{sec result}

In the notation of the previous subsection, we have the vector bundle
\[
\Bigwedge_{J} T(G/H) \to G/H,
\]
with $J$ as in Lemma \ref{lem J}.
As in the discrete series case, this vector bundle has the natural $K$-equivariant Clifford action of Example \ref{ex S J}. Let $\rho^M$ be half the sum of the roots in $R^+_M$.
Consider the one-dimensional representation
\[
\C_{\lambda - \rho^M} \boxtimes \chi_M
\]
of $H_M$. We extend it to a representation of $H$ by letting $A$ act trivially.
This induces the line bundle
\[
L_{\lambda- \rho^M, \chi_M} := G\times_{H} (\C_{\lambda - \rho^M} \boxtimes \chi_M) \to G/H.
\]

Let $\xi \in \kt^*_M$ be any regular element that is dominant with respect to $R^+_M$. We may, and will, assume that the elements $\zeta \in \ka$ and $\xi \in \kt^{*}_M \cong \kt_M$ are chosen so that the element $\xi + \zeta \in \kh$ is regular for the roots of $(\kg^{\C}, \kh^{\C})$.
Define the map 
\[
\Phi\colon G/H \to \kk^* = \kk
\]
by
\[
\Phi(gH) = \Ad^*(g)(\xi + \zeta)|_{\kk}.
\]
(Here we identify $\zeta \in \ka$ with the dual element in $\ka^*$ via the chosen inner product.) The map $\Phi$ is the moment map in the sense of symplectic geometry for the action by $K$ on the coadjoint orbit $\Ad^*(G)(\xi + \zeta)$. 
This map is taming for that action, by Proposition 2.1 in \cite{Paradan99}. (This is true because $\xi + \zeta$ is a regular element, so $\Ad^*(G)(\xi + \zeta)$ is a closed coadjoint orbit.)

Our main result is the following realisation of $\pi|_K$. Recall that
\beq{eq pi basic}
\pi = \Ind_P^G(\pi^{M}_{\lambda, R^+_M, \chi_M} \otimes e^{\nu} \otimes 1_N).
\eeq
\begin{theorem} \label{main theorem}
We have
\[
 \pi|_K = (-1)^{\dim(M/K_M)/2}
\indx_K\bigl(\Bigwedge_{J} T(G/H) \otimes L_{\lambda- \rho^M, \chi_M}, \Phi \bigr). 
\]
\end{theorem}
Note that if $\pi$ belongs to the discrete series or limits of discrete series,  then $M = G$, $H = T$ and $\kn^- = \kn^+ = \{0\}$. So if we take $\xi = \lambda$, then Theorem \ref{main theorem} reduces to Theorem \ref{thm Paradan}. The case of limits of discrete series was not treated in \cite{Paradan03}, because the focus there was to obtain a multiplicity formula for $K$-types of the discrete series. But the techniques used there apply directly to this case.

\begin{remark} \label{rem basic}
The only property of the representation $\pi$ that we will use to prove Theorem \ref{main theorem} is that it is of the form \eqref{eq pi basic}. This is true if $\pi$ is tempered
 by Theorem \ref{thm KZ}. Since the restriction to $K$ of the representation on the right hand side of \eqref{eq pi basic} is independent of the parameter $\nu \in (\ka^{\C})^*$, even for non-imaginary $\nu$, Theorem  \ref{main theorem} in fact applies to every representation $\pi$ equal to the right hand side of \eqref{eq pi basic}, for $\nu \in (\ka^{\C})^*$; i.e.\ to every standard representation. By the Langlands classification of admissible representations (see for example Theorem 14.92 in \cite{KnappBook}), this implies that the restriction to $K$ of every admissible representation is a quotient of an index as in Theorem \ref{main theorem}.
%
\end{remark}

\begin{remark}
The sign $(-1)^{\dim(G/K)/2}$ in Theorem \ref{main theorem} can be absorbed into the Clifford module used. For example, under regularity conditions on $\lambda$ and $\nu$, one can choose the $\Spinc$-structure on $G/H$ whose determinant line bundle is a prequantum line bundle for a relevant coadjoint orbit. However, we prefer the explicit Clifford module in Theorem \ref{main theorem}. This  also avoids references to symplectic geometry, since $\Spinc$-geometry is more relevant here.
We discuss this in more detail in  Subsection \ref{sec orbit}.
\end{remark}

\subsection{Relation with the orbit method and geometric quantisation} \label{sec orbit}

In a companion paper  \cite{PHYSSY2}, we use Theorem \ref{main theorem} together with a \emph{$\Spinc$-quantisation commutes with reduction} result, Theorem 3.10 in \cite{HS16}, to obtain a geometric formula for the multiplicities of the $K$-types of $\pi$. Note that in Theorem \ref{main theorem}, there is a large amount of freedom in choosing the elements $\xi \in \kt^*_M$ and $\zeta \in \ka$. To obtain the multiplicity formula in  \cite{PHYSSY2}, we will need to make more specific choices. 

Let $\rho^G$ be half the sum of the positive roots compatible with $R_M^+$ and $\Sigma^+$. Consider the element
\[
\eta := (\lambda + \nu + \rho^G|_{\kt_M} - \rho^M)/i \quad \in \kh^*.
\]
Then, modulo a factor $i$, $\eta$ is the infinitesimal character of $\pi$, shifted by $\rho^G|_{\kt_M} - \rho^M$.
Suppose that $\eta$ is a regular element of $\kh$. We will see in Section 2.2 of \cite{PHYSSY2} that a natural choice is to take $\xi = \lambda +  \rho^G|_{\kt_M} - \rho^M$ and $\zeta$ dual to $\nu/i$. Then $\Phi$ is the map
\beq{eq phi coadj}
G/H \xrightarrow{\cong} \Ad^*(G)\eta \hookrightarrow \kg^* \to \kk^*.
\eeq
This is a moment map in the symplectic sense for the action by $K$ on the coadjoint orbit $\Ad^*(G)\eta$. 

However, to obtain a multiplicity formula for the $K$-types of $\pi$ from Theorem \ref{main theorem} via the {quantisation commutes with reduction} principle, it is better to use a $\Spinc$-approach to geometric quantisation. Indeed, already in the case of the discrete series, the natural complex structure on $\Ad^*(G)\eta = \Ad^*(G)\lambda = G/T$ is not compatible with the natural symplectic form, and the line bundle $L_{\lambda - \rho^M, \chi_M} = G\times_T \C_{\lambda - \rho}$ is not a prequantum line bundle for that symplectic structure. See Section 1.5 of \cite{Paradan03}. For this reason, it is useful that, for these choices of $\xi$ and $\zeta$, the map $\Phi$ is a moment map in the $\Spinc$-sense, for the $\Spinc$-structure with the spinor bundle $\Bigwedge_{J} T(G/H) \otimes L_{\lambda- \rho^M, \chi_M}$ in Theorem \ref{main theorem}. This is shown in Proposition 2.4 and Lemma 4.5 in \cite{PHYSSY2}. The quantisation commutes with reduction principle was shown to have a natural place in $\Spinc$-geometry by Paradan and Vergne \cite{Paradan14a}.

The Clifford module $\Bigwedge_{J} T(G/H) \otimes L_{\lambda- \rho^M, \chi_M}$ used in Theorem \ref{main theorem} corresponds to a $K$-equivariant $\Spinc$-structure. The manifold $G/H$ does not admit a $G$-equivariant $\Spinc$-stucture in general, because the action by $G$ on $G/H$ is not proper if $H$ is noncompact. Indeed, the manifold then does not even admit a $G$-invariant Riemannian metric. In the context of the orbit method, the $\Spinc$-structure with spinor bundle 
\[
\Bigwedge_{J} T(G/H) \otimes L_{\lambda- \rho^M, \chi_M} \otimes (G\times_H \C_{\nu})
\]
on $G/H \cong \Ad^*(G)\eta$
is a natural one to use. This is $K$-equivariantly isomorphic to the $\Spinc$-structure with spinor bundle $\Bigwedge_{J} T(G/H) \otimes L_{\lambda- \rho^M, \chi_M}$ in Theorem \ref{main theorem}, and we leave out the factor $(G\times_H \C_{\nu})$ to simplify our arguments. See Lemma 4.5 in \cite{PHYSSY2}. (Including that factor would lead to a factor $\C_{\nu}$ in Proposition \ref{prop rewrite piK} below, which is trivial as a representation of $H_M$ so would not change that result.)

The use of index theory of Dirac operators deformed by vector fields induced by moment maps (known as Kirwan vector fields), as in Theorem \ref{main theorem}, is quite common and natural in geometric quantisation. Already in the compact case, such deformations were used to prove that quantisation commutes with reduction in \cite{Paradan01, Zhang98}. In the noncompact case, such deformed Dirac operators are not just used to prove this result, but also to define geometric quantisation. This was done in Vergne's conjecture \cite{Vergne06} for actions by compact groups on noncompact symplectic manifolds, and its generalisation and proof by Ma and Zhang \cite{Zhang14} and later by Paradan \cite{Paradan11} (and even later in \cite{HS17}). In Vergne's conjecture and Paradan's proof, and in Ma and Zhang's proof, two different but equivalent definitions of the equivariant index of deformed Dirac operators were used to the one we use here. The same deformed Dirac operators were used to prove that quantisation commutes with reduction for actions by compact groups on noncompact $\Spinc$-manifolds in \cite{HS16}, and for proper actions by noncompact groups in \cite{Mathai14, Mathai13,  HS18, Mathai10}.

If $\eta$ is not regular, then the map \eqref{eq phi coadj} is still a moment map in the $\Spinc$-sense, but the first arrow is a fibre bundle rather than a diffeomorphism. In this case, the map \eqref{eq phi coadj} does not have the properties that the map $\Phi$ needs to have for Theorem \ref{main theorem} to hold. In  \cite{PHYSSY2}, we will deform the map $\Phi$ in Theorem \ref{main theorem} to a taming, proper $\Spinc$-moment map to obtain a geometric expression for multiplicities of $K$-types of $\pi$.

In representation theory, such representations with singular parameters are viewed as being associated to nilpotent coadjoint orbits. Via the deformation we use in the singular case, we lose the link with this aspect of the orbit method, but we gain a geometric formula for multiplicities of $K$-types.

\subsection{Example: $G = \SL(2,\R)$}\label{sec SL2}

To illustrate Theorem \ref{main theorem}, we work out the case where $G = \SL(2,\R)$ and $K = \SO(2)$. This group $G$ has three kinds of tempered representations: the (holomorphic and antiholomorphic) discrete series, the limits of the discrete series, and the (spherical and nonspherical) unitary principal series.

\subsubsection{The discrete series}

To realise the restrictions of the  {discrete series} to $K$, we take $H = T = \SO(2)$, so $M = G$. Let $\alpha$ be the root that maps $\mattwo{0}{-1}{1}{0}$ to $2i$. 
Let $n \in \{1, 2, 3, \ldots\}$, and let $D_n^{\pm}$ be the discrete series representation with Harish--Chandra parameter 
$\lambda = \pm n\alpha/2$. Then $\rho^M = \rho_{\lambda} = \pm \alpha/2$. We take $\xi = \lambda$. In this case $\ka = \{0\}$, so $\zeta = 0$. Then the map $\Phi$ is defined by $\Phi(gT) = \Ad^*(g)\lambda|_{\kk}$, for $g \in G$. This is the identification of $G/T$ with the elliptic orbit $\Ad^*(G)\lambda$, followed by restriction to $\kk$. We may use the orbit through any positive multiple of $\lambda$ here, but this choice will turn out to be the relevant one for the computation of multiplicities of $K$-types.

Then the complex structure $J_{\kg/\kt}$ is determined by the isomorphism
$
\kg/\kt \cong \kg^{\C}_{\pm\alpha} = \C E_{\pm\alpha},
$
where
\[
E_{\alpha} = \frac{1}{2}\mattwo{1}{-i}{-i}{-1}; \quad E_{-\alpha} = \frac{1}{2}\mattwo{1}{i}{i}{-1}.
\]
The torus $T$ acts on this space with weight $\pm\alpha$. We will write $\C_l$ for the unitary irreducible representation of the circle with weight $l \in \Z$. Then
  $\kg/\kh = \C_{\pm2}$, and $\Bigwedge_{J_{\kg/\kh}}\kg/\kh = \C_0 \oplus \C_{\pm 2}$.

We have $Z_M = Z_G = \{\pm I\} \subset T$, so we have to take $\chi_M = e^{\lambda - \rho^M}|_{Z_M}$. Then
\[
\C_{\lambda - \rho^M} \boxtimes \chi_M \cong \C_{\lambda-\rho^M} = \C_{\pm(n-1)},
\]
via the map $z_1 \otimes z_2 \mapsto z_1z_2$.
So the line bundle $L_{\lambda - \rho^M, \chi_M}$ 
equals
\[
L_{\lambda - \rho^M, \chi_M} =  G\times_T \C_{\pm(n-1)}.
\]
Theorem \ref{main theorem} now states that
\[
D_n^{\pm}|_K = -\indx_K\bigl(G\times_T (\C_{\pm(n-1)} \oplus \C_{\pm(n+1)}), \Phi \bigr).
\]

\subsubsection{Limits of discrete series}

For the {limits of the discrete series} $D_0^{\pm}$, we also have $H = T = \SO(2)$ and $M = G$. Now $\lambda = 0$ and $R^+_M = \{\pm \alpha\}$. We cannot take $\xi = \lambda = 0$, so we take $\xi = \pm \alpha/2$ (we can replace this element by  any positive multiple). As before, we have $\zeta = 0$, so for all $g \in G$, we have $\Phi(gT) = \Ad^*(g)\xi$. This is the identification of $G/T$ with the elliptic orbit $\pm\Ad^*(G)\alpha/2$, followed by restriction to $\kk$. Note that this is the same orbit as the one used for $D^{\pm}_1$; this shift will be important in \cite{PHYSSY2}. 
Now we have $\rho^M = \pm \alpha/2$, so, analogously to the discrete series case,
\[
L_{\lambda - \rho^M, \chi_M} = G\times_T \C_{\mp 1}.
\]
Theorem \ref{main theorem} therefore yields
\[
D_0^{\pm}|_K = -\indx_K\bigl(G\times_T (\C_{\mp 1} \oplus \C_{\pm 1}), \Phi \bigr).
\]

\subsubsection{The principal series}

For the {unitary principal series} $P^{\pm}_{i\nu}$, where $\nu \geq 0$ for the spherical principal series $P^+_{i\nu}$ and $\nu > 0$ for the nonspherical principal series $P^-_{i\nu}$, we have 
\[
H =  \Bigl\{\mattwo{x}{0}{0}{x^{-1}}; x\not=0 \Bigr\}.
\]
Then $M = H_M = \{\pm I\}$. Now
\[
\kg/\kh \cong \kn^- \oplus \kn^+,
\]
where $\kn^{\pm} = \R E^{\pm}$, with 
\[
E^- = \mattwo{0}{0}{1}{0}; \quad  E^+ = \mattwo{0}{1}{0}{0}.
\]
Now $\kt_M = 0$, so
 $\xi = 0$. We take $\zeta = \mattwo{1}{0}{0}{-1}$. Then for all $g \in G$,
$
\Phi(gH) = \Ad^*(g)\zeta|_{\kk}.
$
This is the identification of $G/H$ with the hyperbolic orbit $\Ad^*(G)\zeta$, followed by restriction to $\kk$. Note that we use the same orbit for all principal series representations.

Furhermore,
\[
\ad(\zeta)E^{\pm} = \pm 2E^{\pm}
\]
and $\theta E^{\pm} = -E^{\mp}$. So with respect to the basis $\{E^+, E^-\}$, the complex structure $J_{\zeta}$ has the matrix
\[
J_{\zeta} = \mattwo{0}{1}{-1}{0}.
\]
The almost complex structure $J$ on $G/H$ is given by
\[
J(T_{eH}l_{k\exp(Y)} X) = T_{eH}l_{k\exp(Y)} J_{\zeta}X,
\]
for $k \in \SO(2)$, $Y \in \kn^+$ and $X \in \kn^- \oplus \kn^+$.

Let $\chi_{\pm}$ be the representations of $H_M$ on $\C$ defined by $\chi_{\pm}(-I) = \pm 1$. We now have $\lambda = \rho^M = 0$, so
\[
L_{\lambda - \rho^M, \chi_{\pm}} = G\times_H \chi_{\pm},
\]
where $H$ acts on $\chi_{\pm}$ via 
\[
\mattwo{x}{0}{0}{x^{-1}} \mapsto \chi_{\pm}(\sgn(x)I). 
\]
This number equals $1$ for $\chi_+$, and $\sgn(x)$ for $\chi_-$.
By Theorem \ref{main theorem}, we have
\[
P^{\pm}_{i\nu}|_K = \indx_K(\Bigwedge_{J} T(G/H) \otimes (G\times_H \chi_{\pm}), \Phi).
\]

\subsubsection{Explicit computation of the indices for $\SL(2,\R)$}

The proof of Theorem \ref{main theorem} simplifies considerably in the example $G = \SL(2,\R)$. Some important points are still present, however: a partial linearisation of the space $G/H$, and the role of differential operators on vector spaces with index $1$. We give a brief outline of the proof for $\SL(2,\R)$ to illustrate these points. In particular, we give explicit computations of the relevant indices in this case.

In the case of (limits of) discrete series representations, the proof of Theorem \ref{main theorem} is a variation on the proof of Theorem 5.1 in \cite{Paradan03}.
For the representation $D_n^+$, with $n \in \{0,1,2, \ldots\}$, the first step in this proof 
 is to show that 
\beq{eq index ds}
\indx_K(\Bigwedge_{J}G/T \otimes L_{\lambda - \rho^M, \chi_M}, \Phi) = \indx_K(\Bigwedge_{\C} T\C \otimes\C_{\pm(n-1)},\Phi^E),
\eeq
where $K$ acts on $\C$ by rotation with weight $2$, and $\Phi^E\colon \C \to \mathfrak{so}(2) \cong \R$ is the constant map with value $\pm 1$. Here we use that $G/T$ is isomorphic to $\C$ as a complex $K$-manifold, where $K$ acts on $\C$ by rotation with weight $2$. The index on the right hand side of \eqref{eq index ds} was computed in Lemma 6.4 in \cite{Atiyah74} (see also Lemma 5.7 in \cite{Paradan01} and the bottom of page 841 in \cite{Paradan03}), and equals
\[
-\bigoplus_{j=0}^{\infty} \C_{\pm(n+2j+1)}.
\]
Apart from the minus sign, this is the well-known decomposition of $D^+_n|_K$. The minus sign is the factor $(-1)^{\dim(M/K_M)/2}$ in this case.


For the principal series representation $P^{\pm}_{i\nu}$, the space $G/H$ is isomorphic to the cylinder $\C/\Z$ as a complex $K$-manifold, where $K$ acts on $\C/\Z$ by double rotations of the cylinder in the direction $\R/\Z$. 
We will see that 
\[
\indx_K(\Bigwedge_{J} T(G/H) \otimes L_{\lambda - \rho^M, \chi_M}, \Phi) = \indx_K(\Bigwedge_{\C}T(\C/\Z) \otimes \chi_{\pm}, \tilde \Phi),
\]
where for $x, y \in \R$, we have $\tilde \Phi(x+\Z + iy) = y \in \mathfrak{so}(2)$. As a special case of Proposition \ref{prop htp}, we will see that the index on the right hand side equals the $L^2$-kernel of the operator
\[
\mattwo{0}{-\frac{\partial}{\partial y} + y}{\frac{\partial}{\partial y} + y}{0}
\]
on $\bigl(L^2(\C/\Z) \otimes\C^2 \otimes \chi_{\pm}\bigr)^{\{\pm I\}}$. 
This kernel equals
\[
\begin{split}
\bigl( L^2(\R/\Z) \otimes \C\cdot (y \mapsto e^{-y^2/2}) \otimes \chi_{\pm}\bigr)^{\{\pm I\}} 
&\cong (L^2(K) \otimes \chi_{\pm})^{K_M} \\
&= \Ind_{MAN}^G(\chi_{\pm} \otimes e^{i\nu} \otimes 1_N)|_K \\
&= P^{\pm}_{i\nu}|_K.
\end{split}
\]

Since $K$ acts on $\R/\Z$ by rotations with weight $2$, one also finds directly that 
\[\bigl( L^2(\R/\Z) \otimes \C\cdot (y \mapsto e^{-y^2/2}) \otimes \chi_{\pm}\bigr)^{\{ \pm I\}} = 
\left\{ \begin{array}{ll}
\bigoplus_{j \in \Z} \C_{2j} & \text{for $\chi_+$};\\
\bigoplus_{j \in \Z} \C_{2j+1} & \text{for $\chi_-$}.
\end{array}
\right.
\]
This is the usual decomposition of $P^{\pm}_{i\nu}|_K$. (Note that now $(-1)^{\dim(M/K_M)/2} = 1$.)


\section{Linearising the index} \label{sec linearise}

In the rest of this paper, we prove Theorem \ref{main theorem}.
The first main step in this proof  is to show that the index that appears there equals the index of an operator on a partially linearised version of the space $G/H$, see Proposition \ref{prop linearise} below. This partially linearised space $E$ is the total space of a vector bundle over the compact space $K/H_M$. It will then be shown that the relevant index on $E$ equals $\pi|_K$, up to a sign, in Sections \ref{sec ind fibred}--\ref{sec pi K}, see Proposition \ref{prop lin comp}.
That will finish the proof of Theorem \ref{main theorem}.

We continue using the notation and assumptions of Subsections \ref{sec ac str} and \ref{sec result}.

\subsection{The linearised index}

Consider the action by $H_M$ on $K\times (\ks_M \oplus \kn)$ given by
\[
h\cdot (k, X + Y) = (kh^{-1}, \Ad(h)(X+Y)),
\]
for $h \in H_M$, $k \in K$, $X \in \ks_M$ and $Y \in \kn$. Let
\[
E := K\times_{H_M} (\ks_M \oplus \kn)
\]
be the quotient space of this action. This is the partial linearisation of $G/H$ that we will use. (We wil see in Lemma \ref{lem lin GH} that $E$ is $K$-equivariantly diffeomorphic to $G/H$.)

For $X \in \ks_M$ and $Y \in \kn$, we have
\[
T_{[e, X+Y]}E = \kk/\kt_M \oplus \ks_M \oplus \kn,
\]
via the map
\beq{eq def phi0}
(U+\kt_M, V+W) \mapsto \ddt[\exp(tU), X+Y+t(V+W)],
\eeq
for $U \in \kk$, $V \in \ks_M$ and $W \in \kn$. (We spell this out explicitly here, because another identification will also be used in Subsection \ref{sec lin J}.) Using this identification and the first equality in \eqref{eq decomp gh 1}, we obtain
\beq{eq TE}
T_{[e, X+Y]}E = \kg/\kh.
\eeq
On this space, we defined the $H_M$-invariant complex structure $J_{\kg/\kh}$ in \eqref{eq def Jgh}. We will write $J^E$ for the $K$-invariant almost complex structure on $E$ that corresponds to $J_{\kg/\kh}$ via the identification \eqref{eq TE}.

Consider the map $\Phi^E\colon E \to \kk$ given by
\[
\Phi^E[k, X+Y] = \Ad(k)\bigl( \xi -|\ad(\zeta)|^{-1}\ad(\zeta)Y_{\ks} \bigr),
\]
where $k \in K$, $X \in \ks_M$, $Y \in \kn$,  $\zeta$ is as in Subsection \ref{sec ac str}, $\xi$ is as in Subsection \ref{sec result}, and $Y_{\ks}:= \frac{1}{2}(Y - \theta Y)$ is the component of $Y$ in $\ks$. We will see in Lemma \ref{lem muEt} that this map is taming,

Let $V$ be a finite-dimensional representation space of $H_M$. 
 (Later, we will take $V = \C_{\lambda - \rho^M} \boxtimes \chi_M$.) We extend it to a representation of $H$ by letting $A$ act trivially. Then we obtain the associated vector bundles
\[
\begin{split}
L_V &:= G\times_H V \to G/H;\\
L_V^E &:= K\times_{H_M}(\ks_{M} \oplus \kn \times V) \to E.
\end{split}
\]
The linearisation result for the index in Theorem \ref{main theorem} is the following.
\begin{proposition} \label{prop linearise}
We have
\[
\indx_K(\Bigwedge_J T(G/H) \otimes L_V, \Phi) = \indx_K\bigl(\Bigwedge_{J^E} TE \otimes L_V^E,  \Phi^E).
\]
\end{proposition}

\subsection{Linearising $G/H$}

Consider the map
\[
\tilde \Psi\colon K\times (\ks_M \oplus \kn) \to G/H
\]
defined by
\[
\tilde \Psi(k, X+Y) = k\exp(X) \exp(Y)H,
\]
for $k \in K$, $X \in \ks_M$ and $Y \in \kn$.
\begin{lemma}\label{lem lin GH}
The map $\tilde \Psi$ descends to a well-defined, $K$-equivariant diffeomorphism
\[
\Psi\colon E \xrightarrow{\cong} G/H.
\]
\end{lemma}
\begin{proof}
One checks directly that for $k \in K$, $X \in \ks_M$, $Y \in \kn$ and $h \in H_M$,
\[
\tilde \Psi(kh^{-1}, \Ad(h)(X+Y)) = \tilde \Psi(k, X+Y).
\]
So $\tilde \Psi$ indeed descends to a map $\Psi\colon E \to G/H$. That map is $K$-equivariant by definition. To see that $\Psi$ is surjective, we use the fact that the multiplication map \eqref{eq mult map} is a diffeomorphism. If $k \in K$, $X \in \ks_M$, $Y \in \kn$ and $a \in A$, then
\[
k\exp(X)\exp(Y)aH = \Psi[k, X+Y],
\] 
so $\Psi$ is indeed surjective. If $k, k' \in K$, $X, X' \in \ks_M$ and $Y, Y' \in \kn$ satisfy
\[
\Psi[k, X+Y] = \Psi[k', X'+Y'],
\]
then there are $h \in H_M$ and $a \in A$ such that
\[
k'\exp(X')\exp(Y') = k\exp(X)\exp(Y)ha = kh\exp(\Ad(h^{-1})X)\exp(\Ad(h^{-1})Y)a,
\]
so we must have $a = e$, and $[k', X'+Y'] = [k, X+Y]$. Hence $\Psi$ is injective.
\end{proof}

\subsection{Linearising almost complex structures} \label{sec lin J}

The $K$-invariant almost complex structures $J^E$ and $\Psi^*J$ on $E$ are different, but we will show that they are homotopic in a suitable sense. 

Let $X \in \ks_M$ and $Y \in \kn$. The complex stuctures on $T_{[e, X+Y]}E$ defined by $\Psi^*J$ and $J^E$ correspond to $J_{\kg/\kh}$ via two different identifications
\beq{eq phi01}
\varphi^{X+Y}_0, \varphi^{X+Y}_1\colon \kg/\kh = T_{[e,0]}E  \xrightarrow{\cong} T_{[e, X+Y]}E.
\eeq
The map $\varphi^{X+Y}_0$ is the inverse of the map \eqref{eq def phi0}, whereas $\varphi^{X+Y}_1$ is defined by commutativity of the diagram
\[
\xymatrix{
T_{[e,0]}E \ar[rrr]^-{\varphi^{X+Y}_1} \ar[d]_-{T_{[e,0]}\Psi} & & & T_{[e,X+Y]}E \ar[d]^-{T_{[e,0]}\Psi} \\
T_{eH}G/H \ar[rrr]_-{T_{eH}\exp(X)\exp(Y)} & & &T_{\exp(X)\exp(Y)H} G/H.
}
\]
With respect to these two maps, we have
\beq{eq two ac str}
\begin{split}
J^E_{[e, X+Y]} &= \varphi^{X+Y}_0 J_{\kg/\kh} (\varphi^{X+Y}_0)^{-1}; \\
(\Psi^*J)_{[e, X+Y]} &= \varphi^{X+Y}_1 J_{\kg/\kh} (\varphi^{X+Y}_1)^{-1}. 
\end{split}
\eeq

\begin{lemma} \label{lem lin J}
We have
\[
\indx_K\bigl(\Psi^*(\Bigwedge_J T(G/H) \otimes L_V), \Psi^*\Phi \bigr) = \indx_K(\Bigwedge_{J^E} TE \otimes L_V^E, \Psi^*\Phi).
\]
\end{lemma}
\begin{proof}
For $X \in \ks_M$, $Y \in \kn$ and $t \in [0,1]$, we
define the linear isomorphism
\[
\varphi^{X+Y}_t\colon T_{[e,0]}E \xrightarrow{\cong} T_{[e,X+Y]}E
\]
by commutativity of the diagram
\[
\xymatrix{
T_{[e,0]}E \ar[r]^-{\varphi^{X+Y}_t} \ar[d]_-{\varphi^{t(X+Y)}_1}  & T_{[e,X+Y]}E  \\
T_{[e,t(X+Y)]}E & T_{[e,0]}E. \ar[u]_-{\varphi^{X+Y}_0} \ar[l]^-{\varphi^{t(X+Y)}_0}
}
\]
(Note that for $t=0,1$, this definition agrees with the definitions of $\varphi^{X+Y}_0$ and $\varphi^{X+Y}_1$ above.) Define the $K$-invariant almost complex structure $J_t$ on $E$ by the property that
\[
(J_t)_{[e, X+Y]} = \varphi^{X+Y}_t J_{\kg/\kh} (\varphi^{X+Y}_t)^{-1}.
\]
Then by \eqref{eq two ac str}, we have $J_0 = J^E$ and $J_1 = \Psi^* J$. Using this family of almost complex structures, we obtain a homotopy in the sense of Definition \ref{def htp} between
\[
(\Bigwedge_{\Psi^*J} TE \otimes L_V^E, \Psi^*\Phi) \quad \text{and} \quad (\Bigwedge_{J^E} TE \otimes L_V^E, \Psi^*\Phi).
\]
Since $\Psi^*L_V = L_V^E$ and 
\[
\Psi^*\Bigwedge_J T(G/H) = \Bigwedge_{\Psi^*J} TE,
\]
the claim follows from Theorem \ref{thm htp}.
\end{proof}

\subsection{Linearising taming maps} 

To prove Proposition \ref{prop linearise}, it remains to construct a homotopy between the taming maps $\Psi^* \Phi$ and $\Phi^E$. We will do this in a number of stages.

First, for $t \in [0,1]$, we consider the map $\Phi_1^t\colon E \to \kk$, defined by
\beq{eq def mu1t}
\Phi_1^t[k, X+Y] = \Ad(k)(\xi + t[Y_{\kk}, \xi] + [Y_{\ks}, \zeta]),
\eeq
where $k \in K$, $X \in \ks_M$, $Y \in \kn$, $Y_{\kk} = \frac{1}{2}(Y+\theta Y) \in \kk$, and $Y_{\ks} = \frac{1}{2}(Y-\theta Y) \in \ks$.
For $Z \in \kk$, let  $Z^E$ be the vector field on $E$ given by the infinitesimal action by $Z$. In several places, we will use the fact that 
 for all such $X$ and $Y$,
\beq{eq ZE}
Z^E([e, X+Y]) = (Z_{\kt_M^{\perp}} + \kt_M, -[Z_{\kt_M}, X+Y]) \quad \in \kk/\kt_M \oplus \ks_M \oplus \kn = T_{[e,0]}E \cong T_{[e, X+Y]}E,
\eeq
where the last identification is made via the map $\varphi^{X+Y}_0$ in \eqref{eq phi01}. Furthermore, we write $Z = Z_{\kt_M} + Z_{\kt_M^{\perp}}$, where $Z_{\kt_M} \in \kt_M$ and $Z_{\kt_M^{\perp}} \in \kt_M^{\perp}$.
\begin{proposition} \label{prop mu E}
The vector field $v^{\Phi_1^0}$ on $E$  vanishes at the same points as the vector field $\Psi^*v^{\Phi}$. In a neighbourhood of the set of these points, the taming maps $\Phi_1^0$ and $\Psi^*\Phi$ are homotopic in the sense of Definition \ref{def htp}.
\end{proposition}

The proof of this proposition is based on Lemmas \ref{lem vmu1t}--\ref{lem Phimu mu11}. 
In the arguments below, we will use the $K$-invariant Riemannian metric on $E$ induced by the inner product on $\kg$ and the identification $\varphi^{X+Y}_0$ in \eqref{eq phi01}. As noted in Subsection \ref{sec deformed Dirac}, Braverman's index is independent of the Riemannian metric used, as long as it is complete and $K$-invariant. So we are free to use the Riemannian metric most suited to our purposes.


\begin{lemma}\label{lem vmu1t}
For all $X \in \ks_M$ and $Y \in \kn$,
\[
v^{\Phi_1^t}([e,X+Y]) = \bigl( t[Y_{\kk}, \xi] + [Y_{\ks}, \zeta] + \kt_M, [X+Y, \xi]  \bigr).
\]
\end{lemma}
\begin{proof}
By $\ad$-invariance of the Killing form $B$, we have for all $X \in \kt_M$,
\[
-B([Y_{\ks}, \zeta], X) = B(Y_{\ks}, [X,\zeta])=0,
\]
so $[Y_{\ks}, \zeta] \in \kt_M^{\perp}$.
We also have $[Y_{\kk}, \xi] \in \kt_M^{\perp}$. So the claim follows from  \eqref{eq def mu1t} and \eqref{eq ZE}.
\end{proof}

\begin{lemma}\label{lem est v mu1t X}
There is a constant $C>0$ such that for all $t \in [0,1]$, $X \in \ks_M$ and $Y \in \kn$,
\[
\|v^{\Phi_1^t}([e,X+Y])\| \geq C\|X\|.
\]
\end{lemma}
\begin{proof}
Note that $[X,\xi] \in \ks_M$. Furthermore, since the adjoint action by $H_M$ commutes with $A$, it preserves the spaces $\kg_{\beta}$, and hence $\kn$. So $[Y,\xi] \in \kn$. Since the elements $[X,\xi]$ and $[Y,\xi]$ lie in different subspaces of $\kg$, Lemma \ref{lem vmu1t} implies that there is a constant $C_1 > 0$ such that for all $X$ and $Y$ as above,
\beq{eq est X xi}
\|v^{\Phi_1^t}([e,X+Y])\| \geq \|[X+Y, \xi] \| \geq C_1 \|[X,\xi]\|.
\eeq

Now $X \in \ks_{M} \perp \ka$ and also $X \in \ks  \perp \kt_M$. So $X \perp \kh$, which means that it lies in the sum of the root spaces of $(\kg^{\C}, \kh^{\C})$. Write
\[
X = \sum_{\alpha \in R(\kg^{\C}, \kh^{\C})} X_{\alpha},
\]
where $X_{\alpha} \in \kg^{\C}_{\alpha}$. Since $[X,\zeta]=0$, we have
\[
\begin{split}
\| [X, \xi] \|^2 &=  \| [X,\xi + \zeta] \|^2 \\
&=  \sum_{\alpha \in R(\kg^{\C}, \kh^{\C})} | \langle \alpha, \xi + \zeta \rangle|^2 \| X_{\alpha} \|^2 \\
&\geq \min_{\alpha \in R(\kg^{\C}, \kh^{\C})} |\langle \alpha, \xi + \zeta \rangle|^2 \sum_{\alpha \in R(\kg^{\C}, \kh^{\C})}\|X_{\alpha}\|^2 \\	
&= \min_{\alpha \in R(\kg^{\C}, \kh^{\C})} |\langle \alpha, \xi + \zeta \rangle|^2 \|X\|^2.
\end{split}
\]
Since the element $\xi + \zeta \in \kh$ is regular, the factor $\min_{\alpha \in R(\kg^{\C}, \kh^{\C})} |\langle \alpha, \xi + \zeta \rangle|^2$ is positive. Together with \eqref{eq est X xi}, this implies the claim.
\end{proof}

\begin{lemma}\label{lem est v mu1t Y}
There is a constant $C>0$ such that for all $t \in [0,1]$, $X \in \ks_M$ and $Y \in \kn$,
\[
\|v^{\Phi_1^t}([e,X+Y])\| \geq C\|Y\|.
\]
\end{lemma}
\begin{proof}
By Lemma \ref{lem vmu1t}, we have
\beq{eq est vmu1t Y}
\|v^{\Phi_1^t}([e,X+Y])\| \geq \| t[Y_{\kk}, \xi] + [Y_{\ks}, \zeta]  \|.
\eeq
Since $\kt_M$ and $\ka$ commute, there is a simultaneous weight space decomposition of $\kg^{\C}$ for the adjoint action by these algebras. (This is just a different way of writing the root space decomposition of $\kg^{\C}$ with respect to $\kh = \kt_M \oplus \ka$.) The set of nonzero weights of the action by $\ka$ is $\Sigma$. Let $\Xi \subset i\kt_M^*$ be the set of weights of $\kt_M$. For $\beta \in \Sigma$ and $\delta \in \Xi$, let $\kg^{\C}_{\beta, \delta}$ be the corresponding weight space. Write
\[
Y = \sum_{\beta \in \Sigma^+,\,  \delta \in \Xi} Y_{\beta, \delta},
\]
with $Y_{\beta, \delta} \in \kg^{\C}_{\beta, \delta}$. Then, since $\theta \kg^{\C}_{\beta, \delta} \subset \kg^{\C}_{-\beta, \delta}$, we have
\[
[Y_{\kk}, \xi] = \frac{1}{2}  \sum_{\beta \in \Sigma^+,\,  \delta \in \Xi} [Y_{\beta, \delta}+\theta Y_{\beta, \delta}, \xi] = -\frac{1}{2}  \sum_{\beta \in \Sigma^+,\,  \delta \in \Xi} \langle \delta, \xi \rangle (Y_{\beta, \delta}+\theta Y_{\beta, \delta}).
\]
Similarly,
\[
[Y_{\ks}, \zeta] = \frac{1}{2}  \sum_{\beta \in \Sigma^+,\,  \delta \in \Xi} [Y_{\beta, \delta}-\theta Y_{\beta, \delta}, \zeta] = -\frac{1}{2}  \sum_{\beta \in \Sigma^+,\,  \delta \in \Xi} \langle \beta, \zeta \rangle (Y_{\beta, \delta}+\theta Y_{\beta, \delta}).
\]
So for all $t \in [0,1]$,
\[
\begin{split}
\| t[Y_{\kk}, \xi] + [Y_{\ks}, \zeta]  \|^2 &=  \sum_{\beta \in \Sigma^+,\,  \delta \in \Xi} |t\langle \delta, \xi \rangle+\langle \beta, \zeta \rangle|^2 \Bigl\|\frac{1}{2} (Y_{\beta, \delta}+\theta Y_{\beta, \delta}) \Bigr\|^2 \\
&\geq \min_{\beta \in \Sigma^+,\,  \delta \in \Xi} |t\langle \delta, \xi \rangle+\langle \beta, \zeta \rangle|^2  \sum_{\beta \in \Sigma^+,\,  \delta \in \Xi}  \Bigl\|\frac{1}{2} (Y_{\beta, \delta}+\theta Y_{\beta, \delta}) \Bigr\|^2 \\
&= \min_{\beta \in \Sigma^+,\,  \delta \in \Xi} |t\langle \delta, \xi \rangle+\langle \beta, \zeta \rangle|^2 \|Y_{\kk}\|^2.
\end{split}
\]
Now for all $\beta$ and $\delta$ as above, we have $\langle \delta, \xi \rangle \in i\R$, while $\langle \beta, \zeta \rangle \in \R$. So 
\[
|t\langle \delta, \xi \rangle+\langle \beta, \zeta \rangle| \geq |\langle \beta, \zeta \rangle| > 0,
\]
by assumption on $\zeta$. Therefore, \eqref{eq est vmu1t Y} implies that there is a constant $C_1 > 0$ such that for all $t \in [0,1]$, $X \in \ks_M$ and $Y \in \kn$,
\[
\|v^{\Phi_1^t}([e,X+Y])\| \geq C_1 \|Y_{\kk}\|.
\]

The projection map $\frac{1}{2}(1+\theta)\colon \kn \to \kk$ is injective. So there is a constant $C_2 > 0$ such that for all $Y \in \kn$, $\|Y_{\kk}\| \geq C_2 \|Y\|$. This completes the proof. 
\end{proof}

\begin{lemma}\label{lem est tilde ZE}
There is a constant $C>0$ such that for all $Z \in \kk$, $X \in \ks_M$ and $Y \in \kn$,
\[
\|Z^E([e, X+Y])\| \leq C(1 + \|X+Y\|) \|Z\|.
\]
\end{lemma}
\begin{proof}
By \eqref{eq ZE}, we have
\[
\|Z^E([e, X+Y])\|^2 = \|Z_{\kt_M^{\perp}}\|^2 + \|[Z_{\kt_M}, X+Y]\|^2. 
\]
For a constant $C_1>0$, this is at most equal to
\[
\|Z_{\kt_M^{\perp}}\|^2 + C_1\|Z_{\kt_M}\|^2 \|X+Y\|^2 \leq (1+C_1)(1+ \|X+Y\|^2)\|Z\|^2.
\]
\end{proof}

\begin{lemma}\label{lem Phimu mu11}
We have
\[
\Psi^*\Phi([e,X+Y]) = \Phi_1^1([e,X+Y]) +  \mathcal{O}(\|X+Y\|^2).
\]
as $X+Y\in \ks_M \oplus \kn$ goes to  $0$.
\end{lemma}
\begin{proof}
We have
\[
\begin{split}
\Psi^*\Phi ([e,X+Y]) &= \bigl(\exp(X)\exp(Y)(\xi+ \zeta)\bigr)|_{\kk} \\
	&= \sum_{l, m = 0}^{\infty} \frac{1}{l! m!} (\ad(X)^l \ad(Y)^m (\xi+ \zeta))|_{\kk} \\
	&= (\xi+ \zeta)|_{\kk} + [X,\xi+\zeta]|_{\kk} + [Y,\xi+ \zeta]|_{\kk} +  \mathcal{O}(\|X+Y\|^2).
\end{split}
\]
Now $(\xi+ \zeta)|_{\kk} = \xi$, $[X, \xi + \zeta]|_{\kk} = [X,\zeta]=0$, and 
\[
[Y,\xi+ \zeta]|_{\kk} = [Y_{\kk}, \xi] + [Y_{\ks}, \zeta].
\]
\end{proof}

\noindent \emph{Proof of Proposition \ref{prop mu E}.}
By Lemmas \ref{lem est v mu1t X} and \ref{lem est v mu1t Y}, the vector field $v^{\Phi^t_1}$ vanishes precisely at the set $K\times_{H_M} \{0\} \subset E$. The map $\Phi$ is a moment map for the action by $K$ on $G/H = \Ad^*(G)(\xi + \zeta)$ with respect to the Kirillov--Kostant symplectic form. Therefore, the vector field $v^{\Phi}$ is the Hamiltonian vector field of the function $\frac{1}{2}\|\Phi\|^2$. By Proposition 2.1 in \cite{Paradan03}, this vector field therefore vanishes precisely at the set $\{kH; k \in K\} \subset G/H$. So $v^{\Psi^*\Phi} = \Psi^*v^{\Phi}$ vanishes at the set
\[
\Psi^{-1}(\{kH; k \in K\} ) = K\times_{H_M}\{0\}.
\]

Let $C_1>0$ be as the constant $C$ in Lemma \ref{lem est tilde ZE}.  Then by Lemma \ref{lem Phimu mu11}, we have for all $X \in \ks_M$ and $Y \in \kn$,
\beq{eq est mu11}
\begin{split}
\|v^{\Phi_1^1}([e,X+Y]) - v^{\Psi^*\Phi} [e, X+Y]\| &\leq C_1(1+ \|X+Y\|) \|\Phi_1^1([e,X+Y]) - {\Psi^*\Phi} ([e, X+Y])   \| \\
&= \mathcal{O}(\|X+Y\|^2),
\end{split}
\eeq
as $X+Y \to 0$. By Lemmas \ref{lem est v mu1t X} and \ref{lem est v mu1t Y}, there is a constant $C_2 > 0$ such that for all $t \in [0,1]$, $X \in \ks_M$ and $Y \in \kn$, 
\[
\|v^{\Phi_1^t}([e,X+Y])\| \geq C_2 \|X+Y\|.
\]
Together with \eqref{eq est mu11}, this implies that in a small enough neighbourhood of $K\times_{H_M}\{0\}$ in $E$, we have
\[
(v^{\Phi^1_1}, v^{\Psi^*\Phi}) \geq 0.
\]
Corollary \ref{cor pos inner prod} implies that $\Phi_1^1$ and $\Psi^*\Phi$ are homotopic in this neighbourhood. Lemmas \ref{lem est v mu1t X} and \ref{lem est v mu1t Y} imply that $\Phi_1^0$ and $\Phi_1^1$ are homotopic, so the claim follows. 
\hfill $\square$

\subsection{Proof of Proposition \ref{prop linearise}} \label{sec proof lin}

We prove Proposition \ref{prop linearise} by combining the earlier results in this section with a last homotopy of taming maps.
\begin{lemma}\label{lem muEt}
The map $\Phi^E$ is taming, and homotopic to the map $\Phi_1^0$ defined in \eqref{eq def mu1t}.
\end{lemma}
\begin{proof}
For $t \in [0,1]$, consider the map $\Phi^E_t\colon E\to \kk$ defined by
\[
\Phi^E_t([k, X+Y]) = \Ad(k)(\xi - |\ad(\zeta)|^{-t}\ad(\zeta)Y_{\ks}).
\]
Then $\Phi^E_0 = \Phi_1^0$ and $\Phi^E_1 = \Phi^E$.

Since $|\ad(\zeta)|^{-t}\ad(\zeta)Y_{\ks} \perp \kt_M$ for all $t$, we have by \eqref{eq ZE},
\[
\begin{split}
\|v^{\Phi^E_t}([e, X+Y]) \|^2 &= \bigl\|  \bigl(|\ad(\zeta)|^{-t}\ad(\zeta)Y_{\ks} + \kt_M, [X+Y, \xi] \bigr)  \bigr\|^2\\
	&\geq \||\ad(\zeta)|^{-t}\ad(\zeta)Y_{\ks}\|^2.
\end{split}
\]
If we write $Y = \sum_{\beta \in \Sigma^+} Y_{\beta}$, with $Y_{\beta} \in \kg_{\beta}$, then
\[
\||\ad(\zeta)|^{-t}\ad(\zeta)Y_{\ks}\|^2 = \sum_{\beta \in \Sigma^+} \frac{|\langle \beta, \zeta \rangle|^2}{|\langle \beta, \zeta \rangle|^{2t}} \bigl\| \frac{1}{2}(Y_{\beta}+\theta Y_{\beta}) \bigr\|^2
\geq \min_{\beta \in \Sigma^+} \frac{|\langle \beta, \zeta \rangle|^2}{|\langle \beta, \zeta \rangle|^{2t}}  \|Y_{\kk}\|^2.
\]
As in the proof of Lemma \ref{lem est v mu1t Y}, we conclude that there is a constant $C_1>0$ such that for all $X$ and $Y$ as above,
\[
\|v^{\Phi^E_t}([e, X+Y]) \| \geq C_1\|Y\|.
\]
Exactly as in the proof of Lemma \ref{lem est v mu1t X}, we also find a constant $C_2>0$ such that for all such $X$ and $Y$,
\[
\|v^{\Phi^E_t}([e, X+Y]) \| \geq C_2\|X\|.
\]
Hence the vector field $v^{\Phi^E_t}$ vanishes precisely at the set $K\times_{H_M}\{0\}$. So $\Phi^E_t$ is taming for all $t$, and the claim follows.
\end{proof}

\noindent \emph{Proof of Proposition  \ref{prop linearise}.}
Successively applying Lemma \ref{lem lin GH}, Lemma \ref{lem lin J}, Propositions \ref{prop excision} and \ref{prop mu E}, and finally Lemma \ref{lem muEt}, we find that
\[
\begin{split}
\indx_K(\Bigwedge_J T(G/H) \otimes L_V, \Phi) &= \indx_K\bigl(\Psi^*(\Bigwedge_J T(G/H) \otimes L_V), \Psi^*\Phi\bigr)\\
	&= \indx_K(\Bigwedge_{J^E} TE \otimes L_V^E, \Psi^*\Phi)\\
	&= \indx_K(\Bigwedge_{J^E} TE \otimes L_V^E, \Phi_1^0)\\
&=\indx_K\bigl(\Bigwedge_{J^E} TE \otimes L_V^E,  \Phi^E).
\end{split}
\]
\hfill $\square$


\section{Indices on fibred products} \label{sec ind fibred}

In Section \ref{sec linearised index}, we will explicitly compute the right hand side of the equality in Proposition \ref{prop linearise}, see
Proposition \ref{prop lin index}. This involves a general result about 
indices of deformed Dirac operators on certain fibred product spaces, Proposition \ref{prop htp}.  Our goal in this section is to prove this result.
In this section and the next, we will need to consider more general deformations of Dirac operators than in \eqref{eq deformed Dirac} in some places.

\subsection{The result}

Let $H<K$ be a closed subgroup. Let $N$ be a complete Riemannian manifold
with an isometric action by $H$.
Let  $M := K\times_H N$. Let $\cS \to M$ be a $K$-equivariant Clifford module. Let $\psi\colon M \to \kk$ be a taming map for the action by $K$ on $M$. Then we have the index
\[
\indx_K(\cS, \psi) \in \hat R(K).
\]
We assume that the \label{page Riem metric}Riemannian metric on $M$ is induced by an $H$-invariant Riemannian metric on $N$ and an $H$-invariant inner product on $\kk$. As noted below Theorem \ref{thm Braverman}, the choice of the Riemannian metric does not influence the index.


Let $\nabla^N$ be an $H$-invariant Hermitian Clifford connection on $\cS|_N \to N$. Let $D^N$ be the Dirac operator
\beq{eq DN}
D^N\colon \Gamma^{\infty}(\cS|_N)\xrightarrow{\nabla^N} \Gamma^{\infty}(T^*N\otimes \cS|_N ) \xrightarrow{c} \Gamma^{\infty}(\cS|_N).
\eeq
The vector field $v^{\psi}$ restricts to a section of $TM|_N$, so we have the endomorphism $c(v^{\psi})|_{N}$ of $\cS|_N$. For any $H$-invariant, nonnegative function $f \in C^{\infty}(N)^H$, consider the operator
\beq{eq DN def}
D^N_{f\psi} := D^N -ifc(v^{\psi})|_N
\eeq
on $\Gamma^{\infty}(\cS|_N)$.
\begin{proposition}\label{prop htp} If $f$ is admissible, then
the operator $D^N_{f\psi}$ is Fredholm on every $H$-isotypical component of $L^2(\cS|_N)$. So it has a well-defined index
\[
\indx_H(D^N_{f\psi}) \in \hat R(H).
\]
Furthermore, for any ${\delta} \in \hat K$,
the multiplicities $m_{\delta}^{\pm}$ of ${\delta}$ in the spaces $(L^2K \otimes \ker_{L^2}(D^N_{f\psi})^{\pm} )^H$ are finite, and we have
\[
\indx_K(\cS, \psi) = (L^2(K) \otimes \indx_H(D^N_{f\psi}))^H.
\]
\end{proposition}
\begin{remark}
Roughly speaking, Proposition \ref{prop htp} plays the role in this paper that Theorem 4.1 in \cite{Atiyah74} plays in \cite{Paradan03}.
\end{remark}


\subsection{A connection}

To prove Proposition \ref{prop htp}, we consider the connection $\nabla^M$ on $\cS$ defined by the properties that it is $K$-invariant, and for all $X \in \kh^{\perp}\subset \kk$, $n \in N$, $v \in T_nN$ and $s \in \Gamma^{\infty}(\cS)$,
\beq{eq def nabla M}
(\nabla^M_{X^M_n + v}s)[e, n] := (\calL_X s)([e,n]) + (\nabla^N_v s|_N)(n).
\eeq
Consider the projections
\[
\begin{split}
p_{K/H}&\colon K\times N \to K/H;\\
p_{N}&\colon K\times N \to N;\\
p_M&\colon K\times N \to M.
\end{split}
\]
Let $\nabla^{T(K/H)}$ and $\nabla^{TN}$ be the Levi--Civita connections on $T(K/H)$ and $TN$, respectively. Then the Levi--Civita connection $\nabla^{TM}$ on $TM$ satisfies
\beq{eq decomp LC}
p_M^* \nabla^{TM} = p_{K/H}^*\nabla^{T(K/H)} + p_{N}^*\nabla^{TN}.
\eeq
(See Lemma 3.4 in \cite{HS16-3}.) 
\begin{lemma}
The relation \eqref{eq def nabla M} indeed defines a well-defined, $K$-invariant, Hermitian Clifford connection $\nabla^M$ on $\cS$.
\end{lemma}
\begin{proof}
Every tangent vector in $T_{[e,n]}M$ can be represented in a unique way as $X^M_n + v$, with $X$ and $v$ as above. To show that the $K$-invariant extension is well-defined, we note that $\nabla^M$ as defined above is $H$-invariant. Indeed, $\nabla^N$ is $H$-invariant by assumption, while for all $h \in H$, $T_nh(X^M_n) = (\Ad(h)X)^M_{hn}$, and
\[
\calL_{\Ad(h)X} = h\calL_X h^{-1}.
\]

To verify the Leibniz rule, note that for all $\varphi \in C^{\infty}(M)$,
\begin{multline*}
(\calL_X \varphi s)([e,n]) + (\nabla^N_v \varphi s|_N)(n) \\= \varphi(n)\bigl( (\calL_X s)([e,n]) + (\nabla^N_v  s|_N)(n) \bigr) + (X^M(f)([e,n]) + v(f)(n))s([e,n]).
\end{multline*}
The property that $\nabla^M$ is Hermitian follows from the facts that $\nabla^N$ is, and that the metric on $\cS$ is $K$-invariant.

It remains to show that $\nabla^M$ is a Clifford connection. Let $X, Y \in \kh^{\perp}$, and let $v$ and $w$ be vector fields on $N$. The space of vector fields on $M$ is isomorphic to the subspace
\[
\Gamma^{\infty}(K\times N, p_N^*TN \oplus p_{K/H}^*T(K/H))^H \subset \Gamma^{\infty}(K\times N, p_N^*TN \oplus p_{K/H}^*T(K/H)).
\]
We work in the larger space on the right hand side, where we have the elements $p_N^*v$ and $p_N^*w$. Connections on vector bundles on $M$ define operators on this larger space via the pullback along $p_M$.
Let  $s \in \Gamma^{\infty}(\cS)$. Since $\calL_X$ commutes with $c(p_N^*w)$ and $p_N^*\nabla^N_{v}$ commutes with $c(Y^M)$, we have
\beq{eq comm conn}
\bigl(\bigl[\nabla^M_{X^M+p_N^*v}, c(Y^M+p_N^*w)\bigr]s\bigr)|_N 
	=  ([\calL_X, c(Y^M)]s)|_N +  [\nabla^N_v,  c(w)](s|_N).
\eeq
Now $\nabla^N$ is a Clifford connection, so by \eqref{eq decomp LC},
 $[\nabla^N_v,  c(w)] = c(\nabla^{TM}_v w)$. Furthermore, equivariance of the Clifford action implies that
\[
[\calL_X, c(Y^M)] = c([X,Y]^M).
\]
Now 
 by \eqref{eq decomp LC}, $[X,Y]^M = \nabla^{TM}_{X^M}Y^M$, so we conclude that the right hand side of \eqref{eq comm conn} equals
\[
(c(\nabla^{TM}_{X^M}Y^M)s)|_N + c(\nabla^{TM}_v w)(s|_N) = \bigl(c(\nabla^{TM}_{X^M + v} (Y^M + p_N^*w))s \bigr)|_N.
\]
Here we used that $\nabla^{TM}_{X^M}w = \nabla^{TM}_{v}Y^M = 0$.
\end{proof}

\subsection{Estimates for Dirac operators}

Let $D^M$ be the Dirac operator on $\Gamma^{\infty}(\cS)$ associated to $\nabla^M$. Write
\[
D^M_{f\psi} := D^M -ifc(v^{\psi}),
\]
for a nonnegative function $f \in C^{\infty}(M)^K = C^{\infty}(N)^H$.
Note that
\beq{eq decomp SM}
\cS = K\times_H (\cS|_N),
\eeq
via the map $[k, x] \mapsto k\cdot x$ for all $k \in K$ and $x \in \cS|_N$.
So
\beq{eq decomp Gamma S}
\Gamma^{\infty}(\cS) = (C^{\infty}(K) \otimes \Gamma^{\infty}(\cS|_N))^H.
\eeq
Let $\{X_1, \ldots, X_l\}$ be an orthonormal basis of $\kh^{\perp}$. Then, with respect to the decomposition \eqref{eq decomp Gamma S}, we have
\beq{eq DM minus DN}
D^M - 1\otimes D^N = \sum_{j=1}^l c(X_j^M)\calL_{X_j}.
\eeq

For ${\delta} \in \hat K$, we denote the ${\delta}$-isotypical subspace of $L^2(\cS)$ by $L^2(\cS)_{\delta}$.
\begin{lemma}\label{lem est comm}
For every ${\delta} \in \hat K$, 
\begin{enumerate}
\item[(a)] the operator $D^M - 1\otimes D^N$ is bounded on $L^2(\cS)_{\delta}$ (with norm depending on $\delta$);
\item[(b)] there is a constant $C^{\delta}>0$ such that
for all $f \in C^{\infty}(M)^K$, we have on $L^2(\cS)_{\delta}$,
\begin{multline*}
-C^{\delta} (1+f\|v^{\psi}\|) \leq \\ 
D^M_{f\psi} (D^M - 1\otimes D^N)+ (D^M - 1\otimes D^N)D^M_{f\psi}  \\
\leq C^{\delta} (1+f\|v^{\psi}\|).
\end{multline*}
\end{enumerate}
\end{lemma}
\begin{proof}
Let ${\delta} \in \hat K$. 
We use \eqref{eq DM minus DN}.
Note that $\calL_{X_j}$ is bounded on $L^2(\cS)_{\delta}$, with norm depending on ${\delta}$. For each $j$, because $X_j \in \kh^{\perp}$, and the Riemannian metric is of the form mentioned on page \pageref{page Riem metric}, we have
\[
\|c(X_j^M)\| = \|X_j^M\| = \|(X_j, 0)\| = \|X_j\|,
\]
which is constant.
So part (a) follows.

For part (b), we note that
 $K$-invariance of $D^M_{f\psi}$ implies that
\beq{eq DM DN 1}
D^M_{f\psi} (D^M - 1\otimes D^N)+ (D^M - 1\otimes D^N)D^M_{f\psi} = \sum_{j=1}^l \bigl(D^M_{f\psi}c(X_j^M)+c(X_j^M) D^M_{f\psi}\bigr)\calL_{X_j}.
\eeq
Now for every $j$, $\calL_{X_j}$ is bounded on $L^2(\cS)_{\delta}$, while
\[
D^M_{f\psi}c(X_j^M)+c(X_j^M)D^M_{f\psi} = D^Mc(X_j^M)+c(X_j^M)D^M -2i f(v^{\psi}, X_j^M).
\]
By a straightforward computation (see (1.26) in \cite{Zhang98} and Lemma 9.2 in \cite{Braverman02}), we have for any local orthonormal frame $\{e_1, \ldots, e_{\dim(M)}\}$ of $TM$, and any $j$,
\[
D^Mc(X_j^M)+c(X_j^M)D^M = i\sum_j c(e_j)c(\nabla^{TM}_{e_j}X_j^M) -2i (\calL_{X_j} + \langle \Phi^{\cS}, X_j\rangle), 
\]
with $\Phi^{\cS}$ as in \eqref{eq def mu S}.

Now for every $j$, we have $\langle \Phi^{\cS}, X_j\rangle = 0$ by definition of $\nabla^M$ in \eqref{eq def nabla M}. And
\[
|(v^{\psi}, X_j^M)| \leq \|v^{\psi}\|
\]
by definition of the Riemannian metric used. Finally, we claim that $\|\nabla^{TM}X_j^M\|$ is bounded. Indeed, recall the form \eqref{eq decomp LC} of the Levi--Civita connection on a fibred product space like $M$.
Let $\Phi^{T(K/H)} \in \End(T(K/H)) \otimes \kk^*$ be such that
for all $Z \in \kk$, 
\beq{eq def mu KH}
 \langle \Phi^{T({K/H})}, Z\rangle =
\nabla^{T(K/H)}_{Z+\kh} - \calL_Z. 
\eeq
We have for all $k \in K$ and $n \in N$,
\[
X_j^M(kn) = (\Ad(k)X_j, 0) \quad \in \kh^{\perp} \oplus T_nN \cong T_{kn}M.
\]
So for all $w \in TN$, we have 
\[
\nabla^{TM}_w X_j^M = (p_{N}^*\nabla^{TN})_w X_j^M = 0.
\]
And for all $Z \in \kk$, and $n \in N$,
\[
(\nabla^{TM}_{Z^M} X_j^M)(n) = \bigl((\calL_Z + \langle \Phi^{\cS}, Z\rangle)X_j^M\bigr)(n) = ([Z, X_j] +  \langle \Phi^{\cS}, Z\rangle)X_j, 0).
\]
This is constant in $n$, so we find that $\|\nabla^{TM}X_j^M\|$ is indeed bounded.

By combining the above arguments, we find that the claim in part (b) is true.
\end{proof}

\subsection{Proof of Proposition \ref{prop htp}}

For $t \in [0,1]$, consider the operator
\[
D_{t, f\psi} := 
D^M_{f\psi} + t(1\otimes D^N - D^M). 
\]
on $\Gamma^{\infty}(\cS)$. We view it as an unbounded operator on $L^2(\cS)$.
\begin{lemma}\label{lem discr spec}
If $f \in C^{\infty}(M)^K$ is admissible, then
for all $t \in [0,1]$, the operator $D_{t, f\psi}^2$ has discrete spectrum on every $K$-isotypical subspace of $L^2(\cS)$.
\end{lemma}
\begin{proof}
We have
\[
D_{t, f\psi}^2 = (D^M_{f\psi})^2 + t^2 (1\otimes D^N - D^M)^2 - t \bigl( D^M_{f\psi} (D^M - 1\otimes D^N)+ (D^M - 1\otimes D^N)D^M_{f\psi} \bigr). 
\]
Let ${\delta} \in \hat K$. Braverman shows in the proof of Theorem 2.9, on page 22 of \cite{Braverman02}, 
there is a constant $C_1^{\delta} > 0$  such that for all $f \in C^{\infty}(M)^K$, we have on $L^2(\cS)_{\delta}$,
\[
(D^M_{f\psi})^2 \geq (D^M)^2 + f^2 \|v^{\psi}\|^2 -C_1^{\delta}(\|df\| \|v^{\psi}\| + fh),
\]
with $h$ as in \eqref{eq def h}. Using Lemma \ref{lem est comm}, we let $C_2^{\delta}$ be the norm of $D^M - 1\otimes D^N$ on $L^2(\cS)_{\delta}$, and $C_3^{\delta}$ as the constant $C^{\delta}$ in part (b) of that lemma.
 Then we find that on $L^2(\cS)_{\delta}$, all $f \in C^{\infty}(M)^K$, and all $t \in [0,1]$,
\begin{multline*}
D_{t, f\psi}^2 \geq (D^M)^2 + f^2 \|v^{\psi}\|^2 -C_1^{\delta}(\|df\| \|v^{\psi}\| + fh) - C_2^{\delta} - C_3^{\delta}(1+f\|v^{\psi}\|) \\
\geq (D^M)^2 + f^2 \|v^{\psi}\|^2 -C^{\delta}(\|df\| \|v^{\psi}\| + fh + 1),
\end{multline*}
for  $C^{\delta} := C_1^{\delta} + C_2^{\delta} + C_3^{\delta}$.
If $f$ is admissible, then the function
\[
 f^2 \|v^{\psi}\|^2 -C^{\delta}(\|df\| \|v^{\psi}\| + fh+1)
\]
goes to infinity as its argument goes to infinity in $M$. This implies the claim.
\end{proof}

\noindent \emph{Proof of Proposition \ref{prop htp}.}
Fix ${\delta} \in \hat K$.  Suppose $f \in C^{\infty}(M)^K$ is admissible. Then by Lemma \ref{lem discr spec}, the operator 
\[
\frac{D_{t, f\psi}}{D_{t, f\psi} + i} 
\]
defines a Fredholm operator on $L^2(\cS)_{\delta}$ for every $t \in [0,1]$. This path of bounded operators is continuous in the operator norm, because \eqref{eq DM minus DN} is bounded on $L^2(\cS)_{\delta}$.
So the operators $D^M_{f\psi} = D_{0, f\psi}$ and $D_{1, f\psi}$ on $L^2(\cS)_{\delta}$ have finite-dimensional kernels, and the same index. 

Because of the form of the isomorphism \eqref{eq decomp SM} and $K$-equivariance of $\psi$ and the Clifford action, the operator $c({v^{\psi}})$ on $\Gamma^{\infty}(\cS)$ corresponds to the operator $1\otimes c(v^{\psi})|_N$ on $\bigl(C^{\infty}(K) \otimes \Gamma^{\infty}(\cS|_N) \bigr)^H$. So $D_{1, f\psi} = 1\otimes D^N_{f\psi}$. 
This operator is Fredholm on
\[
L^2(\cS)_{\delta} \cong {\delta} \otimes ({\delta}^* \otimes L^2(\cS|_N))^H
\]
and equals  $1_{\delta} \otimes 1_{{\delta}^*} \otimes D^N_{f\psi}$  on this space. Now
\[
({\delta}^* \otimes L^2(\cS|_N))^H = \bigoplus_{\delta' \in \hat H} [{\delta}|_H:\delta'] L^2(\cS|_N)_{\delta'}.
\]
So the operator $D^N_{f\psi}$ is Fredholm on $L^2(\cS|_N)_{\delta'}$, and therefore has a well-defined equivarant index in $\hat R(H)$.
Furthermore,
\[
\indx_K(D_{1, f\psi}) = \bigl(L^2(K) \otimes \indx_H(D^N_{f\psi}) \bigr)^H.
\]
\hfill $\square$

\begin{remark}
A variation on this proof of Proposition \ref{prop htp} is to note that Lemma \ref{lem discr spec} implies that the operator $D_{t, f\psi}$ defines a class in the $K$-homology of the group $C^*$-algebra $C^*K$, for all $t \in [0,1]$. 

If $N$ is compact (so we may take $f = 0$), then Proposition \ref{prop htp} is a consequence of homotopy invariance of the index of transversally elliptic operators.
\end{remark}


\section{Computing the linearised index} \label{sec linearised index}

In this section, we prove an explicit expression for the right hand side of the equality in Proposition \ref{prop linearise}.

Let $R_n^+ \subset R^+_M$ be the set of noncompact positive roots of $(\km^{\C}, \kt_M^{\C})$. Then, as complex vector spaces, we have
\[
\ks_M = \bigoplus_{\alpha \in R_n^+} \C_{\alpha}.
\]
Define
\[
\bigl(\Bigwedge_{J_{\ks_M}}\ks_M\bigr)^{-1} := \bigotimes_{\alpha \in R_n^+} \bigoplus_{n=0}^{\infty} \C_{n \alpha} \quad \in \hat R(H_M).
\]
This notation is motivated by the equality
\[
\bigl(\Bigwedge_{J_{\ks_M}}\ks_M\bigr)^{-1} \otimes \Bigwedge_{J_{\ks_M}}\ks_M = \C,
\]
the trivial representation of $H_M$. Let $\rho^M_n$ be half the sum of the roots in $R_n^+$.
\begin{proposition}\label{prop lin index}
We have
\begin{multline*}
\indx_K\bigl(\Bigwedge_{J^E} TE \otimes L_V^E,  \Phi^E\bigr)
 = \\
  (-1)^{\dim(M/K_M)/2}
 \bigl(L^2(K) \otimes\ \C_{2\rho^M_n} \otimes \bigl(\Bigwedge_{J_{\ks_M}}\ks_M\bigr)^{-1} \otimes \Bigwedge_{J_{\kk_M/\kt_M}} \kk_M/\kt_M \otimes V\bigr)^{H_M}.
\end{multline*}
\end{proposition}
In Section \ref{sec pi K}, we will use this proposition to complete the proof of Theorem \ref{main theorem}.


\subsection{A constant admissible function}\label{sec f = 1}

We will apply Proposition \ref{prop htp} and then use a further decomposition of the resulting index to prove Proposition \ref{prop lin index}. This decomposition is possible because the constant function $1$ turns out to be admissible in our setting. We prove this in the current subsection.

Let $W_1$ and $W_2$ be finite-dimensional representation spaces of a group $H$. Consider the $H$-equivariant vector bundle
\[
W_1 \times W_2 \to W_1.
\]
Let $\nabla^{W_2}$ be the connection on this bundle defined by 
\[
\nabla^{W_2}_v s = Ts(v)
\]
for all vector fields $v$ on $W_1$ and all $s \in C^{\infty}(W_1, W_2)$. 
Here $Ts$ denotes the derivative of $s$: at a vector $w_1 \in W_1$, we have
\[
(Ts(v))(w_1) = T_{w_1}s(v(w_1)) \quad \in T_{s(w_1)}W_2 = W_2.
\]
By a straightforward computation, this connection is $H$-invariant. 

Now we take
\[
\begin{split}
W_1 &= \ks_M \oplus \kn;\\
W_2 &= \Bigwedge_{J^E}T_{[e,0]}E \otimes V,
\end{split}
\]
and consider the $H$-invariant connection $\nabla^{\bigwedge_{J^E}T_{[e,0]}E \otimes V}$ on
\[
(\ks_M \oplus \kn) \times \Bigwedge_{J^E}T_{[e,0]}E \otimes V = \bigl(\Bigwedge_{J^E}TE \otimes L_V^E\bigr)|_{\ks_M \oplus \kn} \to \ks_M \oplus \kn,
\]
defined as above. 
(We identify $\ks_M \oplus \kn$ with $\{[e, X+Y]; X \in \ks_M, Y \in \kn\} \subset E$.)
Let $\nabla^E$ be the connection on $\Bigwedge_{J^E}TE \otimes L_V^E$ induced by $\nabla^{\bigwedge_{J^E}T_{[e,0]}E \otimes V}$ as in \eqref{eq def nabla M}.

\begin{proposition} \label{prop f = 1}
The constant function $1$ on $E$ is admissible for the connection $\nabla^E$ and the taming map $\Phi^E$.
\end{proposition}

The claim is that the function
\beq{eq cond adm muE}
\frac{\|v^{\Phi^E}\|^2}{\|\Phi^E \| + \|v^{\Phi^E}\| + \|\nabla^{TE} v^{\Phi^E}\| + \| \langle \Phi^{{\bigwedge_{J^E} TE \otimes L^E_V}}, \Phi^E\rangle \| + 1}
\eeq
goes to infinity as its argument goes to infinity in $M$. Here $\Phi^{{\bigwedge_{J^E} TE \otimes L^E_V}}$ is as in \eqref{eq def mu S}.
By \eqref{eq ZE} and the definition of $\Phi^E$, we have for all $X \in \ks_M$ and $Y \in \kn$,
\beq{eq v mu E}
v^{\Phi^E}([e, X+Y]) = \bigl(|\ad(\zeta)|^{-1}\ad(\zeta)Y_{\ks} + \kt_M, -[\xi, X+Y] \bigr). 
\eeq

\begin{lemma}\label{lem est LC}
We have
\[
\|(\nabla^{TE} v^{\Phi^E})([k, X+Y])\| = \cO(\|X+Y\|)
\]
as $X+Y \to \infty$ in $\ks_M \oplus \kn$.
\end{lemma}
\begin{proof}
Recall the expression \eqref{eq decomp LC} for $\nabla^{TE}$, which we now apply with $H = H_M$ and $N = \ks_M \oplus \kn$. Also recall
 the definition \eqref{eq def mu KH} of $\Phi^{T(K/H)} \in \End(T(K/H_M))\otimes \kk^*$. 
Since $v^{\Phi^E}$ is $K$-invariant, we have for all $Z \in \kk$, and all $X \in \ks_M$ and $Y \in \kn$,
\[
\bigl(p_{K/H_M}^*\nabla^{K/H_M}_{Z+\kt_M}v^{\Phi^E}\bigr)([e,X+Y])=\langle \Phi^{T{(K/H)}}, Z\rangle v^{\Phi^E}([e,X+Y]).
\]
Hence
\[
\|\bigl(p_{K/H_M}^*\nabla^{K/H_M}v^{\Phi^E}\bigr)([e,X+Y])\| = \cO(\|v^{\Phi^E}([e,X+Y])\|) = \cO(\|X+Y\|)
\]
as $X+Y \to \infty$ in $\ks_M \oplus \kn$.

For $X \in \ks_M$ and $Y \in \kn$, we have
\[
\nabla^{\ks_M \oplus \kn}_{X+Y} = X+Y,
\]
where on the right hand side, $X$ and $Y$ are viewed as differential operators on $C^{\infty}(\ks_M \oplus \kn)$. Since \eqref{eq v mu E} is linear in $X$ and $Y$, this implies that $(p_{{\ks_M \oplus \kn}}^*\nabla^{\ks_M \oplus \kn} v^{\Phi^E})([e, X+Y])$ is linear in $X$ and $Y$.
\end{proof}

\begin{lemma} \label{lem est muS}
The function $\| \langle \Phi^{{\bigwedge_{J^E} TE \otimes L^E_V}}, \Phi^E\rangle \|$ is constant.
\end{lemma}
\begin{proof}
If $Z \in \kt_M$, then for all $s \in \Gamma^{\infty}({\bigwedge_{J^E} TE \otimes L^E_V})$,
\[
\bigl(\bigl(\nabla^{\bigwedge_{J^E} TE \otimes L^E_V}_{Z^E} - \calL_Z \bigr)s\bigr)|_{\ks_M \oplus \kn} = 
\nabla_{Z^{\ks_M \oplus \kn}}(s|_{\ks_M \oplus \kn}) - \calL_Z(s|_{\ks_M \oplus \kn}).
\]
Now for $X \in \ks_M$ and $Y \in \kn$, we have
\[
\begin{split}
\calL_Zs([e,X+Y]) &= \ddt \exp(tZ)s([e,\exp(-tZ) (X+Y)]) \\
	&= \ad(Z)(s([e,X+Y])) +  \bigl(T(s|_{\ks_M \oplus \kn})(Z^{\ks_M \oplus \kn})\bigr)(X+Y).
\end{split}
\]
So, for $Z \in \kt_M$, 
\[
\langle \Phi^{\bigwedge_{J^E} TE \otimes L^E_V}, Z\rangle = \ad(Z).
\]
If $Z \in \kt_M^{\perp}$, then the left hand side of this equality equals $0$.
We conclude that for all $X \in \ks_M$ and $Y \in \kn$,
\[
\langle \Phi^{\bigwedge_{J^E} TE \otimes L^E_V}, \Phi^E\rangle ([e,X+Y]) = \ad(\xi),
\]
which is constant in $X$ and $Y$.
\end{proof}

\begin{proofof}{Proposition \ref{prop f = 1}}
By Lemmas \ref{lem est v mu1t X} and \ref{lem est v mu1t Y}, there is a constant $C>0$ such that for all $X \in \ks_M$ and $Y \in \kn$,
\[
\|v^{\Phi^0_1}([e, X+Y])\| \geq C(\|X\| + \|Y\|).
\]
The proofs of these lemmas also directly show that this estimate holds with $\Phi_1^0$ replaced by $\Phi^E$. Furthermore, we have
\[
\| \Phi^E([e,X+Y]) \| = \cO(\|X+Y\|)
\]
as $X+Y \to \infty$ in $\ks_M \oplus \kn$. Together with Lemmas \ref{lem est LC} and \ref{lem est muS}, these estimates imply that \eqref{eq cond adm muE} goes to infinity as its argument goes to infinity.
\end{proofof}

\subsection{Indices on vector spaces}

The starting point of the proof of Proposition \ref{prop lin index} is an application of Proposition \ref{prop htp}. Note that $\ks_M$ is a complex subspace of $\km/\kt_M$ with respect to the complex structure $J_{\km/\kt_M}$; let $J_{\ks_M}$ be the restriction of $J_{\km/\kt_M}$ to $\ks_M$.
Consider the trivial vector bundle
\[
\ks_{M} \times \Bigwedge_{J_{\ks_M}} \ks_M \to \ks_M.
\]
We view this bundle as a nontrivial $H_M$-vector bundle.
The constant map $\Phi_{\ks_M}\colon \ks_M \to \kt_M$ with value $\xi$ is a taming map for the action by $H_M$ on $\ks_M$. Therefore, we have the equivariant index
\[
\indx_{H_M}(\ks_{M} \times \Bigwedge_{J_{\ks_M}} \ks_M, \Phi_{\ks_M}) \quad \in \hat R(H_M).
\]

Next, consider the trivial vector bundle
\[
\kn \times \Bigwedge_{J_{\zeta}}(\kn^- \oplus \kn^+) \to \kn.
\]
(Recall that $\kn = \kn^+$; we use the notation $\kn^+$ where this space appears in the direct sum with $\kn^-$.) This bundle is also a nontrivial $H_M$-vector bundle.  We write $c \circ J_{\zeta}$ for the endomorphism of this bundle give by
\[
c \circ J_{\zeta}(Y) = c(J_{\zeta}Y),
\]
for $Y \in \kn$.
Here $c$ denotes the Clifford action by $\kn^- \oplus \kn^+$ on $\Bigwedge_{J_{\zeta}}(\kn^- \oplus \kn^+)$ as in Example \ref{ex S J}. Note that for $Y \in \kn$, the vector $J_{\zeta}Y \in \kn^-$ does not lie in the tangent space $T_Y \kn$. Hence $Y \mapsto J_{\zeta} Y$ is not a vector field on $\kn$, so the endomorphism $c\circ J_{\zeta}$ is not of the type of the deformation term in 
 \eqref{eq deformed Dirac}. Let $\{Y_1, \ldots, Y_s\}$ be an orthonormal basis of $\kn$. Consider the Dirac operator
 \[
 D^{\kn} := \sum_{j=1}^s Y_j \otimes c(Y_j)
 \]
on 
\[
\Gamma^{\infty}(\kn \times \Bigwedge_{J_{\zeta}}(\kn^- \oplus \kn^+)) = C^{\infty}(\kn) \otimes \Bigwedge_{J_{\zeta}}(\kn^- \oplus \kn^+).
\]
The operator
\beq{eq Dn}
D^{\kn} -ic\circ J_{\zeta}
\eeq
is not of the form \eqref{eq deformed Dirac}, but we will see in Lemma \ref{lem index n} that it has a well-defined $H_M$-equivariant index 
(which is actually just the trivial representation of $H_M$).
\begin{proposition}\label{prop decomp index}
The $L^2$-kernel of the operator \eqref{eq Dn} is finite-dimensional, and
\begin{multline*}
\indx_K\bigl(\Bigwedge_{J^E} TE \otimes L_V^E,  \Phi^E\bigr)
 = \\
 \bigl(L^2(K) \otimes \indx_{H_M}(\ks_{M} \times \Bigwedge_{J_{\ks_M}} \ks_M, \Phi_{\ks_M}) \otimes \indx_{H_M}(D^{\kn} -ic\circ J_{\zeta}) \otimes \Bigwedge_{J_{\kk_M/\kt_M}} \kk_M/\kt_M \otimes V\bigr)^{H_M},
\end{multline*}
where $J_{\kk_M/\kt_M}$ is the restriction of $J_{\km/\kt_M}$ to the complex subspace $\kk_M/\kt_M \subset \km/\kt_M$.
\end{proposition}

We will prove this proposition in the rest of this section. But first, we show how it allows us to prove Proposition \ref{prop lin index}.
\begin{lemma}\label{lem index p}
We have
\[
 \indx_{H_M}(\ks_{M} \times \Bigwedge_{J_{\ks_M}} \ks_M, \Phi_{\ks_M}) = (-1)^{\dim(M/K_M)/2} \C_{2\rho^M_n} \otimes \bigl(\Bigwedge_{J_{\ks_M}}\ks_M\bigr)^{-1},
\]
where $\rho^M_n$ is half the sum of the elements of $R_n^+$.
\end{lemma}
\begin{proof}
See the bottom of page 841 in \cite{Paradan03}, and also Lemma 5.7 in \cite{Paradan01}; these in turn are based on  Proposition 6.2 in \cite{Atiyah74}. In the notation of those references, we use the positive expansion of the inverse, because $(\xi, \alpha) > 0$ for all $\alpha \in R_n^+$. 
\end{proof}

In Subsection \ref{sec index n}, we will compute the index on $\kn$ that occurs in Proposition \ref{prop decomp index}, and reach the following conclusion.
\begin{lemma} \label{lem index n}
We have
\[
\indx_{H_M}(D^{\kn} -ic\circ J_{\zeta}) = \C,
\]
the trivial representation of $H_M$.
\end{lemma}

Combining Proposition \ref{prop decomp index} and Lemmas \ref{lem index p} and \ref{lem index n}, we conclude that Proposition \ref{prop lin index} is true.

\subsection{Decomposing the index on $E$}

Let us prove Proposition \ref{prop decomp index}. In Proposition \ref{prop htp}, we take $H = H_M$, $N = \ks_M \oplus \kn$, $\cS = \Bigwedge_{J^E}TE$, and $\psi = \Phi^E$. By Proposition \ref{prop f = 1}, we may take $f = 1$ in Proposition \ref{prop htp}. That proposition then implies that
\beq{eq decomp ind E}
\indx_K\bigl(\Bigwedge_{J^E} TE \otimes L_V^E,  \Phi^E\bigr) = \bigl(L^2(K) \otimes \indx_{H_M}(D^{\ks_M \oplus \kn}_{\Phi^E}) \bigr)^{H_M},
\eeq
with $D^{\ks_M \oplus \kn}_{f\Phi^E}$ as in \eqref{eq DN def}, defined with respect to the connection $\nabla^{\bigwedge_{J^E}T_{[e,0]}E \otimes V}$ as in Subsection \ref{sec f = 1}.
 The operator $D^{\ks_M \oplus \kn}_{\Phi^E}$ acts on sections of the vector bundle
\[
\begin{split}
\bigl(\Bigwedge_{J^E} TE \otimes L_V^E \bigr)|_{\ks_M \oplus \kn} &= (\ks_M \oplus \kn) \times \bigl(\Bigwedge_{J_{\ks_M}} \ks_M \otimes \Bigwedge_{J_{\zeta}} (\kn^- \oplus \kn^+) \otimes \Bigwedge_{J_{\kk_M/\kt_M}} \kk_M/\kt_M \bigr) \otimes V\\
&=\bigl(\ks_M  \times  \Bigwedge_{J_{\ks_M}} \ks_M\bigr)\boxtimes\bigl(\kn \times \Bigwedge_{J_{\zeta}} (\kn^- \oplus \kn^+) \bigr)\otimes \Bigwedge_{J_{\kk_M/\kt_M}} \kk_M/\kt_M \otimes V.
\end{split}
\]
Here $\boxtimes$ denotes the exterior tensor product of vector bundles, we use graded tensor products everywhere, and, as before,  we identify 
\[
\ks_M \oplus \kn \cong \{[e,X+Y]; X \in \ks_M, Y \in \kn\} \subset E.
\]
 In terms of this decomposition, \eqref{eq v mu E} implies that for all $X \in \ks_M$ and $Y \in \kn$,
\begin{multline} \label{eq cvmuE}
c(v^{\Phi^E})(X+Y) =\\ c( -[\xi,X])\otimes 1_{\bigwedge_{J_{\zeta}} (\kn^- \oplus \kn^+) } \otimes 1_{\bigwedge_{J_{\kk_M/\kt_M}} \kk_M/\kt_M} \otimes 1_V \\
 + 
1_{\bigwedge_{J_{\ks_M}} \ks_M} \otimes 
\bigl(c(-[\xi,Y]) + c( |\ad(\zeta)|^{-1} \ad(\zeta) Y_{\ks}) \bigr)
\otimes 1_{\bigwedge_{J_{\kk_M/\kt_M}} \kk_M/\kt_M} \otimes 1_V.
\end{multline}
Here we used the fact that the element 
\[
|\ad(\zeta)|^{-1} \ad(\zeta) Y_{\ks} \quad \in \kt_M^{\perp}
\]
induces a term  in $\kk/\kt_M$ orthogonal to $\kk_M$, hence in the space $\kk/\kk_M$ that is identified with $\kn^-$. So the Clifford action by this vector only acts nontrivially on $\Bigwedge_{J_{\zeta}} (\kn^- \oplus \kn^+)$, and trivially on $\Bigwedge_{J_{\ks_M}} \ks_M $.

Define the 
map  $v^{\kn}\colon \kn \to \kn^- \oplus \kn^+$  
by
\[
v^{\kn}(Y) = -[\xi,Y] + |\ad(\zeta)|^{-1} \ad(\zeta) Y_{\ks},
\]
for $Y \in \kn$.  
Then \eqref{eq cvmuE} becomes
\begin{multline*}
c(v^{\Phi^E})(X+Y) = \\c( v^{\Phi_{\ks_M}}(X))\otimes 1_{\bigwedge_{J_{\zeta}} (\kn^- \oplus \kn^+) } \otimes 1_{\bigwedge_{J_{\kk_M/\kt_M}} \kk_M/\kt_M} \otimes 1_V\\
 + 
1_{\bigwedge_{J_{\ks_M}} \ks_M} \otimes 
 c(v^{\kn}(Y))
\otimes 1_{\bigwedge_{J_{\kk_M/\kt_M}} \kk_M/\kt_M} \otimes 1_V.
\end{multline*}
(Recall that $\Phi_{\ks_M}$ is the constant map with value $\xi$.)

Let $\{X_1, \ldots, X_r\}$ be an orthonormal basis of $\ks_{M}$. Let $D^{\ks_M}$ be the Dirac operator
\[
D^{\ks_M} := \sum_{j=1}^r X_j \otimes c(X_j)
\]
on 
\[
\Gamma^{\infty}(\ks_M \times \Bigwedge_{J_{\ks_M}}\ks_M) = C^{\infty}(\ks_M) \otimes \Bigwedge_{J_{\ks_M}}\ks_M.
\]
Then
 for the choice of the connection on $(\Bigwedge_{J^E}TE \otimes L_V^E)|_{\ks_M \oplus \kn}$ that we made at the start of Subsection \ref{sec f = 1}, the deformed Dirac operator $D^{\ks_M \oplus \kn}_{\Phi^E}$ as in \eqref{eq DN def} equals
\begin{multline} \label{eq decomp D s n mu E}
D^{\ks_M \oplus \kn}_{\Phi^E} = \\
\bigl(D^{\ks_M} -ic( v^{\Phi_{\ks_M}})\bigr) \otimes  1_{\bigwedge_{J_{\zeta}} (\kn^- \oplus \kn^+) } \otimes 1_{\bigwedge_{J_{\kk_M/\kt_M}} \kk_M/\kt_M} \otimes 1_V\\
+
1_{\bigwedge_{J_{\ks_M}} \ks_M} \otimes \bigl(D^{\kn} -i  c(v^{\kn})\bigr) \otimes 1_{\bigwedge_{J_{\kk_M/\kt_M}} \kk_M/\kt_M} \otimes 1_V.
\end{multline}

\begin{lemma}\label{lem decomp index}
The index $\indx_{H_M}(D^{\kn} -ic(v^{\kn}))$ is well-defined, so is its tensor product with $\indx_{H_M}(\ks_{M} \times \Bigwedge_{J_{\ks_M}} \ks_M, \Phi_{\ks_M})$, and we have
\begin{multline*}
\indx_{H_M}(D^{\ks_M \oplus \kn}_{\Phi^E}) =\\
\indx_{H_M}(D^{\kn} -ic(v^{\kn})) \otimes \indx_{H_M}(\ks_{M} \times \Bigwedge_{J_{\ks_M}} \ks_M, \Phi_{\ks_M})\otimes \Bigwedge_{J_{\kk_M/\kt_M}} \kk_M/\kt_M \otimes V.
\end{multline*}
\end{lemma}
\begin{proof}
Because the tensor products in \eqref{eq decomp D s n mu E} are graded, we have
\begin{multline*} 
\bigl(D^{\ks_M \oplus \kn}_{\Phi^E}\bigr)^2 = \\
\bigl(D^{\ks_M} -ic( v^{\Phi_{\ks_M}})\bigr)^2 \otimes  1_{\bigwedge_{J_{\zeta}} (\kn^- \oplus \kn^+) } \otimes 1_{\bigwedge_{J_{\kk_M/\kt_M}} \kk_M/\kt_M} \otimes 1_V\\
+
1_{\bigwedge_{J_{\ks_M}} \ks_M} \otimes \bigl(D^{\kn} -i  c(v^{\kn})\bigr)^2 \otimes 1_{\bigwedge_{J_{\kk_M/\kt_M}} \kk_M/\kt_M} \otimes 1_V.
\end{multline*}
Since all operators occurring here are nonnegative, this implies that
\[
\ker (D^{\ks_M \oplus \kn}_{\Phi^E}) = \ker (D^{\ks_M} -ic( v^{\Phi_{\ks_M}})) \otimes \ker (D^{\kn} -i  c(v^{\kn})) \otimes  \Bigwedge_{J_{\kk_M/\kt_M}} \kk_M/\kt_M \otimes V.
\]
Here the kernels are $L^2$-kernels, and the equality includes gradings. 
\end{proof}

To prove Proposition \ref{prop decomp index}, we will show that we may replace $D^{\kn} -ic(v^{\kn})$ by $D^{\kn} -ic \circ J_{\zeta}$ in the above result. Consider the 
map  $\tilde v^{\kn}\colon \kn \to \kn^- \oplus \kn^+$
given by
\[
 \tilde v^{\kn}(Y) = |\ad(\zeta)|^{-1}\ad(\zeta)Y_{\ks},
\]
for $Y \in \kn$. 
\begin{lemma}\label{lem JY}
Let $Y \in \kn$. Under the identification
\[
T_{[e,Y]}E \cong T_{[e,0]}E = \ks_{M} \oplus \kn^-\oplus \kn^+ \oplus \kk_M/\kt_M
\]
via the map $\varphi^Y_0$ in \eqref{eq phi01}, we have
\[
\tilde v^{\kn}(Y) = J_{\zeta}Y.
\]
\end{lemma}
\begin{proof}
Let $Y \in \kn$.
Since $\ad(\zeta)$ anticommutes with $\theta$,
we have
\[
 |\ad(\zeta)|^{-1}\ad(\zeta)Y_{\ks} =
\frac{1}{2}(1+\theta) |\ad(\zeta)|^{-1}\ad(\zeta)Y  = \frac{1}{2}(1+\theta)\theta J_{\zeta} Y = \frac{1}{2}(1+\theta) J_{\zeta} Y.
\]
The identification $\kn^- \cong \kk/\kk_M$ is made via the map $\frac{1}{2}(1+\theta)$, so the claim follows.
\end{proof}

%
\begin{proposition}\label{prop indices mun}
The multiplicities of all irreducible representations of $H_M$ in the $L^2$-kernels of the operators
\beq{eq D tilde mu}
D^{\kn} -ic(v^{\kn}) \quad \text{and} \quad D^{\kn} - ic(\tilde v^{\kn})
\eeq
are finite, and we have
\[
\indx_{H_M}(D^{\kn} -ic(v^{\kn} )) = \indx_{H_M}(D^{\kn} -ic(\tilde v^{\kn} )) \quad \in \hat R(H_M).
\]
\end{proposition}
This proposition will be proved in Subsection \ref{sec pf indices mun}.

\medskip
\begin{proofof}{Proposition \ref{prop decomp index}}
By \eqref{eq decomp ind E} and Lemma \ref{lem decomp index}, we have
\begin{multline*}
\indx_K\bigl(\Bigwedge_{J^E} TE \otimes L_V^E,  \Phi^E\bigr)
 = \\
 \bigl(L^2(K) \otimes \indx_{H_M}(\ks_{M} \times \Bigwedge_{J_{\ks_M}} \ks_M, \Phi_{\ks_M}) \otimes \indx_{H_M}(D^{\kn} -ic(v^{\kn})) \otimes \Bigwedge_{J_{\kk_M/\kt_M}} \kk_M/\kt_M \otimes V\bigr)^{H_M}.
\end{multline*}
By Lemma \ref{lem JY}, we have $c(\tilde v^{\kn}) = c\circ J_{\zeta}$, so the proposition follows from 
Proposition \ref{prop indices mun}. 
\end{proofof}

\subsection{The index on $\kn$} \label{sec index n}

The proof of Lemma \ref{lem index n} is a direct computation on a vector space. 
 Let $e_0$ and $e_1$ be the generators of the complex exterior algebra $\Bigwedge_{\C} \C = \C \oplus \C$ in degrees $0$ and $1$, respectively. The Clifford action $c$ by $\C$ on $\Bigwedge_{\C} \C $ as in Example \ref{ex S J} now has the form
\[
\begin{split}
c(z) e_0 &= ze_1; \\
c(z) e_1 &= -\bar z e_0,
\end{split}
\]
for $z \in \C$. Let $c\circ i$ be the endomorphism of the trivial bundle $\R \times \Bigwedge_{\C} \C \to \R$ given by
\[
(c\circ i) (x) = c(ix)
\]
for $x \in \R$. Consider the Dirac operator
\[
D^{\R} := c(1)\frac{d}{dx}
\]
on $C^{\infty}(\R, \Bigwedge_{\C} \C)$.
\begin{lemma}\label{lem index R}
The kernel of the operator $D^{\R} -ic\circ i$ intersected with $L^2(\R, \Bigwedge_{\C} \C)$ is one-dimensional, and spanned in degree zero by the function
\[
x\mapsto e^{-x^2/2}.
\]
\end{lemma}
\begin{proof}
Let $s \in C^{\infty}(\R, \Bigwedge_{\C} \C)$, and write $s = s_0 e_0 + s_1 e_1$, for (complex-valued) $s_0, s_1 \in C^{\infty}(\R)$. The equation $(D^{\R} -ic\circ i)s=0$ then becomes
\[
\begin{split}
s_0'+xs_0 &= 0;\\
s_1'-xs_1 &= 0.
\end{split}
\]
This is to say that there are complex constants $a$ and $b$ such that for all $x \in \R$,
\[
\begin{split}
s_0(x) &= ae^{-x^2/2};\\
s_1(x) &= be^{x^2/2}.
\end{split}
\]
\end{proof}

Lemma \ref{lem index R} directly generalises to higher dimensions. Let $n \in \N$, and consider the endomorphism $c\circ i$ of the trivial bundle
\beq{eq bundle Rn}
\R^n \times \Bigwedge_{\C} \C^n \to \R,
\eeq
given by $c\circ i(x) = c(ix)$, for $x \in \R^n$. Let $\{v_1, \ldots, v_n\}$ be the standard basis of $\R^n$, with corresponding coordinates $(x_1, \ldots, x_n)$. Consider the Dirac operator
\[
D^{\R^n} := \sum_{j=1}^n c(v_j)\frac{\partial}{\partial x_j}
\]
on $C^{\infty}(\R^n, \Bigwedge_{\C} \C^n)$. We now also consider an isometric action by a compact Lie group $H$ on $\R^n$ (with the Euclidean metric), and suppose this action lifts to the bundle \eqref{eq bundle Rn} so that $c$ is $H$-invariant. 
\begin{lemma}\label{lem index Rn}
The kernel of the operator $D^{\R^n} -ic\circ i$ intersected with $L^2(\R^n, \Bigwedge_{\C} \C^n)$ is one-dimensional in even degree, zero-dimensional in odd degree, and the action by $H$ on this kernel is trivial.
\end{lemma}
\begin{proof}
It follows from Lemma \ref{lem index R} that $\ker(D^{\R^n} -ic\circ i) \cap L^2(\R^n, \Bigwedge_{\C} \C^n)$ is one-dimensional, spanned in degree zero by the function
\[
x\mapsto e^{-\|x\|^2/2}.
\]
Since $H$ acts isometrically, this function is $H$-invariant.
\end{proof}

\begin{proofof}{Lemma \ref{lem index n}}
Via a complex-linear isomorphism $\kn^- \oplus \kn^+ \cong \C^s$ mapping $\kn^+$ to $\R^s$, the operator  $D^{\kn} -ic\circ J_{\zeta}$ corresponds to the operator $D^{\R^n} -ic\circ i$ in Lemma \ref{lem index Rn}. That lemma therefore implies that the kernel of the operator $D^{\kn} -ic\circ J_{\zeta}$ intersected with $L^2(\kn, \Bigwedge_{J_{\zeta}} \kn^- \oplus \kn^+)$ is one-dimensional in even degree, zero-dimensional in odd degree, and the action by $H_M$ on this kernel is trivial.
\end{proofof}

\begin{remark}
From a higher viewpoint, Lemma \ref{lem index n} is a special case of Bott periodicity. Indeed, the operator $D^{\kn} -ic\circ J_{\zeta}$ is of Callias-type \cite{Bunke95, Kucerovsky01}. Therefore, it is Fredholm, and by Proposition 2.18 in \cite{Bunke95} or Lemma 3.1 in \cite{Kucerovsky01}, its index is the Kasparov product of the classes
\[
\begin{split}
[D^{\kn}] &\in KK_H(C_0(\kn), \C)\quad \text{and}\\ 
[ic\circ J_{\zeta}] &\in KK_H(\C, C_0(\kn)).
\end{split}
\]
The latter class is the Bott generator, and the fact that the index of $D^{\kn} -ic\circ J_{\zeta}$ is the trivial representation of $H_M$ means that $[D^{\kn}]$ is the inverse of that class.
\end{remark}

\subsection{Proof of Proposition \ref{prop indices mun}} \label{sec pf indices mun}

We cannot directly apply Theorem \ref{thm htp} to prove Proposition \ref{prop indices mun}, because the operators involved are not of the kind considered in Subsection \ref{sec deformed Dirac}. (The deformation terms are not given by Clifford multiplication by vector fields on $\kn$.) So we need slightly modified arguments.

Let $\|\cdot\|_{\kn}$ be the norm function on $\kn$. 
\begin{lemma}\label{lem square Rn}
We have
\[
(D^{\kn}-ic\circ J_{\zeta})^2 \geq (D^{\kn})^2 + \|\cdot\|_{\kn}^2 - \dim(\kn).
\]
\end{lemma}
\begin{proof}
By a direct computation, we have
\[
(D^{\R}-ic\circ i)^2 = -\frac{d^2}{dx^2} + \mattwo{x^2-1}{0}{0}{x^2+1} \geq-\frac{d^2}{dx^2} + x^2-1,
\]
with respect to the basis $\{e_0, e_1\}$ of $\Bigwedge_{\C}\C$.
By factorising $D^{\R^n} - ic\circ i$, we deduce that
\[
(D^{\R^n}-ic\circ i)^2  \geq- \sum_{j=1}^n\frac{\partial ^2}{\partial x_j^2} + \|x\|^2-n.
\]
\end{proof}

For $t \in [0,1]$, set
\[
v^t := tv^{\kn} + (1-t)\tilde v^{\kn}.
\]
Define the map  $\xi^{\kn}\colon \kn \to \kn^- \oplus \kn^+$ by
\[
\xi^{\kn}(Y)=-[\xi,Y],
\]
for $Y \in \kn$.
\begin{lemma}\label{lem square mut}
We have for all $t \in [0,1]$, for any orthonormal basis $\{Y_1, \ldots, Y_{s}\}$ of $\kn$,
\[
(D^{\kn} -ic(v^t))^2 = (D^{\kn}-ic\circ J_{\zeta})^2 + t^2 \|\xi^{\kn}\|^2 - it \sum_{j} c(Y_j)c(Y_j(\xi^{\kn})) +2it \calL_{\xi}.
\]
\end{lemma}
\begin{proof}
By Lemma \ref{lem JY}, we have for all $Y \in \kn$,
\[
v^t(Y) = t\xi^{\kn}(Y)+J_{\zeta}Y.
\]
So
\[
(D^{\kn} -ic(v^t))^2 =(D^{\kn}-ic\circ J_{\zeta})^2 + t^2 \|\xi^{\kn}\|^2 -it\bigl( (D^{\kn}-ic\circ J_{\zeta})c(\xi^{\kn}) + c(\xi^{\kn})(D^{\kn}-ic\circ J_{\zeta}) \bigr).
\]
Now if $Y \in \kg_{\beta}$, for $\beta \in \Sigma^+$, then $J_{\zeta}Y \in \kg_{-\beta}$, whereas $\xi^{\kn}(Y) \in \kg_{\beta}$. Since the adjoint action by $\ka$ is symmetric, it follows that 
$
J_{\zeta}Y \perp \xi^{\kn}(Y).
$
So
\[
(c\circ J_{\zeta})  c(\xi^{\kn}) +  c(\xi^{\kn}) (c\circ J_{\zeta}) = 0.
\]
Hence
\[
(D^{\kn}-ic\circ J_{\zeta})c(\xi^{\kn}) + c(\xi^{\kn})(D^{\kn}-ic\circ J_{\zeta}) = D^{\kn}c(\xi^{\kn}) + c(\xi^{\kn})D^{\kn}.
\]
Let $\{Y_{1}, \ldots, Y_{s}\}$ be an orthonormal basis of $\kn$. Then by a direct computation,
\[
 D^{\kn}c(\xi^{\kn}) + c(\xi^{\kn})D^{\kn} = \sum_{j} c(Y_j)c(Y_j(\xi^{\kn})) -2 \calL_{\xi}.
\]
\end{proof}

\begin{lemma}\label{lem square munt}
For all $\delta \in \hat K$, there is a constant $C_{\delta} > 0$ such that
on $L^2(\kn, \Bigwedge_{J_{\zeta}}(\kn^- \oplus \kn^+))_{\delta}$, we have for all $t \in [0,1]$,
\[
(D^{\kn} -ic(v^t))^2 \geq (D^{\kn})^2 + \|\cdot \|_{\kn}^2 + t^2\|\xi^{\kn}\|^2 - C_{\delta}.
\]
\end{lemma}
\begin{proof}
By Lemmas \ref{lem square Rn} and \ref{lem square mut}, we have for all $t \in [0,1]$,
\[
(D^{\kn} -ic(v^t))^2 \geq  (D^{\kn})^2 + \|\cdot\|_{\kn}^2 - \dim(\kn) + t^2 \|\xi^{\kn}\|^2 - it \sum_{j} c(Y_j)c(Y_j(\xi^{\kn})) +2it \calL_{\xi}.
\]
On isotypical components, the Lie derivative $\calL_{\xi}$ is bounded. And for all $j$,
\[
c(Y_j)c(Y_j(\xi^{\kn})) \geq -\|Y_j(\xi^{\kn})\|.
\]
Since $\xi^{\kn} = -\ad(\xi)$ is  a linear map from $\kn$ to itself, its derivatives are constant.
\end{proof}

\begin{proofof}{Proposition \ref{prop indices mun}}
Fix $\delta \in \hat K$. Let $C_{\delta}$ be as in Lemma \ref{lem square munt}.
For $t \in \R$, set $D_t := D^{\kn} - ic(v^t)$. This operator is essentially self-adjoint by a finite propagation speed argument, see e.g.\ Proposition 10.2.11 in \cite{Higson00}. Also, it is $K$-equivariant, so it preserves isotypical components. Hence the operator
\[
 F_t := \frac{D_t}{D_t + iC_{\delta}} \quad \in  \cB(L^2(\kn, \Bigwedge_{J_{\zeta}}(\kn^- \oplus \kn^+))_{\delta})
\]
is well-defined. This operator is Fredholm for all $t$ by Lemma \ref{lem square munt}. It has the same kernel as $D_t$, so the claim follows if we prove that the path $t \mapsto F_t$ is continuous in the operator norm.

To prove this, we first show that $c(\xi^{\kn})(D_t + iC_{\delta})^{-1}$ is bounded. Let $C_{\xi}>0$ be a constant such that 
\[
\|\xi^{\kn}\| \leq C_{\xi}\|\cdot\|_{\kn}.
\]
(Explicitly, we can take $C_{\xi} = |\langle \alpha, \xi \rangle|$ for the root $\alpha$ of $(\kg^{\C}, \kh^{\C})$ for which this number is maximal.) Let $s \in \Gamma^{\infty}_c(\kn, \Bigwedge_{J_{\zeta}}(\kn^- \oplus \kn^+))_{\delta}$. Then by Lemma \ref{lem square munt} and self-adjointness of $D_t$,
\[
\|c(\xi^{\kn})s\|_{L^2}^2 \leq C_{\xi}^2 \bigl\| \|\cdot\|_{\kn}s \bigr\|_{L^2}^2 \leq C_{\xi}^2(\|D_t s\|_{L^2}^2 + C_{\delta}\|s\|)_{L^2}^2 =C_{\xi}^2 \|(D_t + iC_{\delta})s\|_{L^2}^2.
\]
In other words,
\[
\|c(\xi^{\kn})(D_t + iC_{\delta})^{-1}(D_t + iC_{\delta})s\|_{L^2}^2 \leq C_{\xi}^2 \|(D_t + iC_{\delta})s\|_{L^2}^2.
\]
Since $(D_t + iC_{\delta})$ is invertible, every element of $L^2(\kn, \Bigwedge_{J_{\zeta}}(\kn^- \oplus \kn^+))_{\delta}$ is of the form $(D_t + iC_{\delta})s$ for some $s \in (D_t + iC_{\delta})s$. So we conclude that $c(\xi^{\kn})(D_t + iC_{\delta})^{-1}$ is indeed bounded, with norm at most $C_{\xi}^2$.

Now let $t_1, t_2 \in \R$. Then
\[
F_{t_1} - F_{t_2} = C_{\delta} (t_2 - t_1)(D_{t_1}+iC_{\delta})^{-1}c(\xi^{\kn})(D_{t_2}+iC_{\delta})^{-1}.
\]
By the above argument, the operator 
\[
(D_{t_1}+iC_{\delta})^{-1}c(\xi^{\kn})(D_{t_2}+iC_{\delta})^{-1}
\]
is bounded uniformly in $t_1$ and $t_2$. We conclude that the path $t \mapsto F_t$ is indeed continuous in the operator norm.
\end{proofof}

\section{Rewriting $\pi|_K$} \label{sec pi K}

In this section, we rewrite the restricted representation $\pi|_K$ as follows. Let $\rho^M_c$ and $\rho^M_n$ be half the sums of the compact and noncompact roots in $R^+_M$, respectively.
\begin{proposition} \label{prop rewrite piK}
We have
\[
\pi|_K = \bigl(L^2(K)  \otimes \bigl(\Bigwedge_{J_{\ks_M}}\ks_M\bigr)^{-1} \otimes \Bigwedge_{J_{\kk_M/\kt_M}} \kk_M/\kt_M \otimes \C_{\lambda - \rho^M_c +\rho^M_n} \boxtimes \chi_M\bigr)^{H_M}.
\]
\end{proposition}
Together with Propositions \ref{prop lin index} and \ref{prop linearise}, this will allow us to prove Theorem \ref{main theorem} in Subsection \ref{sec pf main thm}.

\subsection{Rewriting Blattner's formula}

In this section only, suppose that $G$ is a Lie group satisfying the assumptions made in Subsection \ref{sec limit ds}. In particular, we suppose that $G$ has discrete series and limits of discrete series representations.

Paradan gave the following reformulation of Blattner's formula. 
\begin{lemma}[Paradan]\label{lem Blatt Par}
Let $\pi^{G_0}_{\lambda, R^+_G}$ be a discrete series or limit of discrete series representation of the connected group $G_0$, with parameters $(\lambda, R^+_G)$ as in Subsection \ref{sec limit ds}. Then 
\[
\pi^{G_0}_{\lambda, R^+_G}|_{K_0} = \bigl( L^2(K_0) \otimes (\Bigwedge_{J_{\ks}} \ks)^{-1} \otimes \Bigwedge_{J_{\kk/\kt}} \kk/\kt \otimes \C_{\lambda - \rho_c + \rho_n}\bigr)^T.
\]
\end{lemma}
For discrete series representations, this is Lemma 5.4 in \cite{Paradan03}. The proof given there extends directly to limits of discrete series, because Blattner's formula applies to those representations as well. (See the bottom of page 131 and the top of page 132 in \cite{HS75}.)

We will need a generalisation of this result to disconnected groups. Let 
\beq{eq Blatt disconn}
\pi^G_{\lambda, R^+_G, \chi} = \Ind_{G_0 Z_G}^G(\pi^{G_0}_{\lambda, R^+_G}\boxtimes \chi)
\eeq
be a discrete series or limit of discrete series representation of $G$, as in \eqref{eq lds G}.
\begin{proposition}\label{prop Blatt disconn}
We have
\begin{equation}\label{eq Blatt disconn2}
\pi^G_{\lambda, R^+_G, \chi}|_K =  \bigl( L^2(K) \otimes (\Bigwedge_{J_{\ks}} \ks)^{-1} \otimes \Bigwedge_{J_{\kk/\kt}} \kk/\kt \otimes \C_{\lambda - \rho_c + \rho_n} \boxtimes \chi \bigr)^{TZ_G}.
\end{equation}
\end{proposition}
\begin{proof}
Note that
\[
\begin{split}
\pi^G_{\lambda, R^+_G, \chi}|_K &= \Ind_{G_0 Z_G}^G(\pi^{G_0}_{\lambda, R^+_G} \boxtimes \chi)|_K \\
&= \Ind_{K_0 Z_G}^K((\pi^{G_0}_{\lambda, R^+_G}\boxtimes \chi)|_{K_0Z_G})\\
&= \Ind_{K_0 Z_G}^K((\pi^{G_0}_{\lambda, R^+_G}|_{K_0}\boxtimes \chi).
\end{split}
\]
Consider the element
\[
V := (\Bigwedge_{J_{\ks}} \ks)^{-1} \otimes \Bigwedge_{J_{\kk/\kt}} \kk/\kt \otimes \C_{\lambda - \rho_c + \rho_n}  \in \hat R(T).
\]
then by Lemma \ref{lem Blatt Par},  and by Lemma \ref{lem ind tensor} below, we have
\[
\begin{split}
 \Ind_{K_0 Z_G}^K((\pi^{G_0}_{\lambda, R^+_G}|_{K_0}\boxtimes \chi)
&= \Ind_{K_0 Z_G}^K(\Ind_{T}^{K_0}(V) \boxtimes \chi) \\
&=  \Ind_{K_0 Z_G}^K(\Ind_{TZ_G}^{K_0Z_G}(V \boxtimes \chi)) \\
&= \Ind_{TZ_G}^K(V \boxtimes \chi),
\end{split}
\]
which is the right hand side of \eqref{eq Blatt disconn2}.
\end{proof}

\begin{remark}
Note that $Z_G$ acts on $\ks$ and $\kk/\kt$ trivially, hence the character $e^{2 \rho_n}$ is equal to the trivial character when restricted to $T \cap Z_G$. Therefore $e^{\lambda - \rho_c + \rho_n} = e^{\lambda - \rho}$ when restricted to $T \cap Z_G$ and the right hand side of \eqref{eq Blatt disconn2} makes sense given the condition \eqref{eq chi} on $\chi$.
\end{remark}

\begin{lemma}\label{lem ind tensor}
For all $V \in \hat R(T)$, 
  we have 
    \begin{equation}\label{eq ind tensor}
       \Ind_{T}^{K_0}(V) \boxtimes \chi =  \Ind_{TZ_G}^{K_0Z_G}(V \boxtimes \chi).
    \end{equation}
\end{lemma}
\begin{proof}
  Since $K_0 / T \cong K_0 Z_G / T Z_G$, it can be shown that the restrictions of the representations on both sides of \eqref{eq ind tensor} to $K_0$ are equal to $\Ind_{T}^{K_0}(V)$. Since $Z_G$ commutes with all elements in $K_0$, the restrictions of the representations on both sides to $Z_G$ are equal to $\chi$. The lemma follows.
\end{proof}

\subsection{Proof of Theorem \ref{main theorem}} \label{sec pf main thm}


We are now prepared to prove Proposition \ref{prop rewrite piK} and Theorem \ref{main theorem}.
First, we consider the setting of Proposition \ref{prop rewrite piK}.
We have
\[
\pi|_K =  \Ind_P^G(\pi^{M}_{\lambda, R^+_M, \chi_M} \otimes e^{\nu} \otimes 1_N)|_K = \bigl(L^2(K) \otimes \pi^{M}_{\lambda, R^+_M, \chi_M}\bigr)^{K_M}.
\]
By Proposition \ref{prop Blatt disconn}, this equals
\begin{multline*}
\Bigl(L^2(K) \otimes   \bigl( L^2(K_M) \otimes (\Bigwedge_{J_{\ks_M}} \ks_M)^{-1} \otimes \Bigwedge_{J_{\kk_M/\kt_M}} \kk_M/\kt_M \otimes \C_{\lambda - \rho^M_c + \rho^M_n} \boxtimes \chi_M \bigr)^{T_MZ_M} \Bigr)^{K_M}\\
=
  \bigl( L^2(K) \otimes (\Bigwedge_{J_{\ks_M}} \ks_M)^{-1} \otimes \Bigwedge_{J_{\kk_M/\kt_M}} \kk_M/\kt_M \otimes \C_{\lambda - \rho^M_c + \rho^M_n} \boxtimes \chi_M \bigr)^{T_MZ_M}.
\end{multline*}
Proposition \ref{prop rewrite piK} now follows from \eqref{eq HM}.

Combining Propositions \ref{prop lin index} and \ref{prop rewrite piK}, we obtain a realisation of $\pi|_K$ as an index on $E$.
\begin{proposition}\label{prop lin comp}
In Proposition \ref{prop linearise}, if we take $V = \C_{\lambda- \rho^M} \boxtimes \chi_M$, then
\[
 \pi|_K = (-1)^{\dim(M/K_M)/2}
\indx_K\bigl(\Bigwedge_{J^E} TE \otimes L_V^E,  \Phi^E\bigr).
\]
\end{proposition}
By combining this with Proposition \ref{prop linearise}, we conclude that Theorem \ref{main theorem} is true.

\bibliographystyle{plain}
\bibliography{mybib}

\end{document}